\documentclass{amsart}

\usepackage{amsmath, amssymb, amsthm,enumitem,bm,pict2e,xcolor,hyperref,tikz-cd,indentfirst,xparse}
\usepackage[foot]{amsaddr}

\hypersetup{
    colorlinks=true,
    linkcolor=blue,
    filecolor=magenta,      
    urlcolor=cyan,
}

\theoremstyle{definition}
\newtheorem{theorem}{Theorem}
\newtheorem*{theorem*}{Theorem}
\numberwithin{theorem}{subsection}
\newtheorem{definition}[theorem]{Definition}
\newtheorem*{definition*}{Definition}
\newtheorem{proposition}[theorem]{Proposition}
\newtheorem{lemma}[theorem]{Lemma}
\newtheorem{remark}[theorem]{Remark}
\newtheorem{example}[theorem]{Example}
\newtheorem{cor}[theorem]{Corollary}

\numberwithin{equation}{section}

\newtheorem{innercustomthm}{Theorem}
\newenvironment{customthm}[1]
  {\renewcommand\theinnercustomthm{#1}\innercustomthm}
  {\endinnercustomthm}

\DeclareMathOperator{\Gr}{Gr}
\DeclareMathOperator{\initial}{in}

\newcommand{\la}{\lambda}
\newcommand{\om}{\omega}

\newcommand{\cL}{\mathcal L}
\newcommand{\cJ}{\mathcal J}
\newcommand{\cO}{\mathcal O}
\newcommand{\cC}{\mathcal C}

\newcommand{\cR}{\mathcal R}

\newcommand{\bC}{\mathbb{C}}
\newcommand{\bP}{\mathbb{P}}
\newcommand{\bR}{\mathbb{R}}
\newcommand{\bZ}{\mathbb{Z}}
\newcommand{\bd}{{\mathbf d}}
\newcommand{\one}{\mathbf 1}

\DeclareFontFamily{U}{mathx}{\hyphenchar\font45}
\DeclareFontShape{U}{mathx}{m}{n}{
      <5> <6> <7> <8> <9> <10>
      <10.95> <12> <14.4> <17.28> <20.74> <24.88>
      mathx10
      }{}
\DeclareSymbolFont{mathx}{U}{mathx}{m}{n}
\DeclareMathSymbol{\bigtimes}{1}{mathx}{"91}

\mathcode`l="8000
\begingroup
\makeatletter
\lccode`\~=`\l
\DeclareMathSymbol{\lsb@l}{\mathalpha}{letters}{`l}
\lowercase{\gdef~{\ifnum\the\mathgroup=\m@ne \ell \else \lsb@l \fi}}%
\endgroup

\mathchardef\newbracket=\mathcode`)
\mathcode`)="8000
\begingroup
\catcode`) \active
\gdef){\nolinebreak\newbracket}
\endgroup

\mathchardef\newcomma=\mathcode`,
\mathcode`,="8000
\begingroup
\catcode`, \active
\gdef,{\nolinebreak\newcomma}
\endgroup

\title{Relative poset polytopes and semitoric degenerations}

\author{Evgeny Feigin$^{1}$}
\email{evgfeig@gmail.com}

\author{Igor Makhlin$^{2}$}
\email{iymakhlin@gmail.com}

\address{$^1$School of Mathematical Sciences, Tel Aviv University, Tel Aviv
69978, Israel}
\address{$^2$Technische Universit\"at Berlin}

\dedicatory{Dedicated to the memory of our teacher and colleague Ernest Borisovich Vinberg.}

\begin{document}

\maketitle

\begin{abstract} 
The two best studied toric degenerations of the flag variety are those given by the Gelfand--Tsetlin and FFLV polytopes. Each of them degenerates further into a particular monomial variety which raises the problem of describing the degenerations intermediate between the toric and the monomial ones. Using a theorem of Zhu one may show that every such degeneration is semitoric with irreducible components given by a regular subdivision of the corresponding polytope. This leads one to study the parts that appear in such subdivisions as well as the associated toric varieties. It turns out that these parts lie in a certain new family of poset polytopes which we term \textit{relative poset polytopes}: each is given by a poset and a weakening of its order relation. In this paper we give an in depth study of (both common and marked) relative poset polytopes and their toric varieties in the generality of an arbitrary poset. We then apply these results to degenerations of flag varieties. We also show that our family of polytopes generalizes the family studied in a series of papers by Fang, Fourier, Litza and Pegel while sharing their key combinatorial properties such as pairwise Ehrhart-equivalence and Minkowski-additivity.
\end{abstract}


\section*{Introduction}
Toric degenerations of algebraic varieties is a popular subject in modern mathematics due to various applications in representation theory (\cite{Ca,FaFL,FFL2}), algebraic geometry (\cite{AB,A,Ma}), mirror symmetry (\cite{BCKS,CDK}), etc. A particular subgenre in this field, to which this paper may be attributed, is the study of semitoric degenerations: those all of whose irreducible components are toric (\cite{Ch,KM,MG,Fu}).
The combinatorial nature of toric geometry provides various tools for the study of algebraic varieties, representations of Lie groups and other mathematical objects of interest. The central role in the combinatorial approach is played by convex geometric structures such as polytopes, cones and fans. 

Several methods of obtaining flat degenerations of algebraic varieties are known. One of those is the Gr\"obner degeneration where one considers the defining ideal of a projective embedding and then an initial ideal thereof. In particular, toric degenerations are given by binomial initial ideals. One may then obtain deeper degenerations considering initial ideals of the binomial ideal, all the way down to monomial ones.


There are many papers treating toric degenerations of Grassmannians and flag varieties for simple Lie groups  (see e.g.\ \cite{GL,KM,FaFL,Kn}). The two most studied degenerations correspond to the Gelfand--Tsetlin (\cite{GT}) and Feigin--Fourier--Littelmann--Vinberg (standardly abbreviated as FFLV, \cite{FFL1}) polytopes. Both of them are given by initial binomial ideals of the Pl\"ucker ideal and in both cases one has a distinguished monomial initial ideal naturally stemming from the corresponding notions of semistandard Young tableaux (\cite{MS,FL,M}). It is therefore a natural goal to describe the Gr\"obner degenerations intermediate between the toric and monomials ones. By an elegant (and underappreciated!) result of Zhu (\cite{Zh}) in a fairly general situation such degenerations are semitoric with components given by a regular subdivision of the toric degeneration's polytope (in the sense of~\cite{GKZ}). This brings us to the problem of studying the subdivisions arising in the aforementioned situations. In the Gelfand--Tsetlin case the situation is simpler, since it can be shown that all parts arising in the subdivisions are themselves marked order polytopes (this was the subject of an earlier version of this paper written before the authors learned of Zhu's general result). However, the FFLV case requires us to introduce the notion of relative poset polytopes which subsequently allows us to treat both cases in a uniform fashion. 

Before giving the definition of relative poset polytopes let us recall that the Gelfand--Tsetlin and FFLV polytopes are, respectively, the marked order polytope and marked chain polytope associated with the so-called Gelfand--Tsetlin poset (\cite{ABS}). 
Order and chain polytopes are the basic examples of poset polytopes, the study of which was initiated in~\cite{stan}. Various generalizations of these notions were given in~\cite{ABS,FF,FFLP,FFP,HLLLT}. However, these families of polytopes turn out to be insufficient for our goal and thus relative poset polytopes are introduced.
Our strategy is as follows: we first give a detailed study of relative poset polytopes for arbitrary posets and then apply these results to describe two families of semitoric degenerations of flag varieties. In short, we show that 
\textit{all Gr\"obner degenerations between either the Gelfand--Tsetlin or FFLV toric degeneration and the respective monomial degeneration are semitoric with their irreducible components being toric varieties of relative poset polytopes}.     


Consider a finite set $P$ with two partial orders $<$ and $<'$ such that $<'$ is weaker than $<$. Let $\cJ$ and $\cJ'$ be the sets of order ideals (lower sets) in $P$ with respect to $<$ and $<'$. In particular, $\cJ$ is naturally embedded into $\cJ'$. Here is the main definition of the paper. \begin{definition*}
The \textit{relative poset polytope} $\cR(P,<,<')$ is the convex hull of the points $\mathbf 1_{\max_{<'}J}\in\bR^P$ for all $J\in\cJ$. 
\end{definition*}
In particular, if $<=<'$, then $\cR(P,<,<')$ is the chain polytope of $(P,<)$ and if $<'$ is trivial, then $\cR(P,<,<')$ coincides with the order polytope, hence relative poset polytopes interpolate between the two. Imposing the restriction that $\cJ$ is closed under the Hibi-Li binary operation $*'$ corresponding to the order $<'$ (\cite{HL}) one can show that relative poset polytopes enjoy many nice properties. In particular, we prove the following (where the third property is of key importance for our goals).
\begin{customthm}{A}\label{thma}
Relative poset polytopes enjoy the following properties:
\begin{itemize}
\item all $\cR(P,<,<')$ are normal,
\item every $\cR(P,<,<')$ is Ehrhart-equivalent to $\cO(P,<)$ and $\cC(P,<)$,
\item any Gr\"obner degeneration of the toric variety of $\cR(P,<,<')$ that degenerates further into the monomial variety of $(P,<)$ is semitoric with its components given by relative poset polytopes forming a subdivision of $\cR(P,<,<')$.
\end{itemize}
\end{customthm}

The theorem above can be applied to describe two families of semitoric degenerations of every Grassmann variety (see Theorems \ref{grass1} and \ref{grass2}). However, the generality of these polytopes is not enough to cover the case of arbitrary partial flag varieties. To address this issue we introduce the notion of marked relative poset polytopes $\cR_\la(P,<,<')$ which additionally depend on a subset $P^*\subset P$ of marked elements and a vector $\la\in\bR^{P^*}$. These polytopes are defined by first considering the \textit{fundamental} case $\la\in\{0,1\}^{P^*}$ (in which they are essentially relative poset polytopes of subposets of $(P,<)$) and then defining the general case as a Minkowski sum of fundamental ones. We study these polytopes by applying results obtained in the non-marked case, in particular, we generalize the properties in Theorem~\ref{thma} to the marked setting. As an application of this construction we realize our initial goal and obtain an explicit description of two large families of semitoric degenerations for every type A partial flag variety. 

    
\begin{customthm}{B}
All Gr\"obner degenerations of a partial flag variety intermediate between the Gelfand--Tsetlin (resp.\ FFLV) toric degeneration and the corresponding monomial degeneration are semitoric with their toric components given by parts in a subdivision of the corresponding Gelfand--Tsetlin (resp.\ FFLV) polytope. Each of these parts is itself a marked relative poset polytope.
\end{customthm}

We also point out that this study was partially motivated by the series of papers~\cite{FF,FFLP,FFP} where another family of poset polytopes termed \textit{marked chain-order polytopes} is defined and studied. In particular, several of the key combinatorial properties of marked relative polytopes (such as pairwise Ehrhart equivalence) are shared by marked chain-order polytopes. In the end of this paper we show that the notion of marked relative poset polytopes generalizes that of marked chain-order polytopes. Furthermore, we explain why marked chain-order polytopes are, in a particular sense, insufficient to describe the semitoric degenerations of flag varieties studied here.

The paper is organized as  follows. In Section \ref{preliminaries} we review the notions and results used in this paper. In particular, we recall the definitions of Hibi and Hibi--Li varieties and state Zhu's theorem. In Section \ref{nonmarked} we define and study relative poset polytopes. We then apply the general results to describe semitoric degenerations of Grassmann varieties. In Section \ref{marked} we generalize the definitions and results in the previous section to the case of marked relative poset polytopes. This generalization is used to describe the semitoric degenerations of all partial flag varieties. Finally, we compare our construction with that of Fang--Fourier--Litza--Pegel. 

\textbf{Acknowledgments.} 
The authors would like to thank Xin Fang for valuable comments.

\section{Preliminaries}\label{preliminaries}

\subsection{Gr\"obner degenerations and Gr\"obner fans}

For a positive integer $N$ consider the polynomial ring $S=\bC[X_1,\ldots,X_N]$. For a set of vectors $d^j\in\bZ_{\ge0}^N$ consider a polynomial \[p=\sum_j c_j X^{d^j}\in S\] with $c_j\neq 0$ (for tuples $a_1,\dots,a_n$ and $b_1,\dots,b_n$ we use the notation $a^b=a_1^{b_1}\dots a_n^{b_n}$). Let $(,)$ be the standard scalar product in $\bR^N$. Consider a vector $w\in\bR^N$ and let $\min_j (w,d^j)=m$. The corresponding \textbf{initial part} of $p$ is then \[\initial_w p=\sum_{j|(w,d^j)=m} c_j X^{d^j}.\] In other words, we define a grading on $S$ by setting the grading of $X_i$ equal to $w_i$ and then take the nonzero homogeneous component of $p$ of the least possible grading. For an ideal $I\subset S$ its \textbf{initial ideal} $\initial_w I$ is the linear span of $\{\initial_w p, p\in I\}$ which is easily seen to be an ideal in $S$.

An important property of this construction is that the algebra $S/\initial_w I$ is a flat degeneration of the algebra $S/I$, i.e.\ there exists a flat $\mathbb A^1$-family of algebras with all general fibres isomorphic to $S/I$ and the special fibre isomorphic to $S/\initial_w I$.
\begin{theorem}[see, for instance,~{\cite[Corollary 3.2.6]{HH}}]\label{flatfamily}
There exists a flat $\bC[t]$-algebra $\tilde S$ such that $\tilde S/\langle t-a\rangle\simeq S/I$ for all nonzero $a\in\bC$ while $\tilde S/\langle t\rangle\simeq S/\initial_w I$.
\end{theorem}

Now let $I$ be homogeneous with respect to the standard grading by total degree and let $\initial_w I$ be a radical ideal. The former obviously implies that $\initial_w I$ is also homogeneous while the latter implies that $I$ is itself radical (see~\cite[Proposition 3.3.7]{HH}). In this case $I$ is the vanishing ideal of a variety $X\subset\bP(\bC^N)$ and $\initial_w I$ is the vanishing ideal of a variety $X^w\subset\bP(\bC^N)$. We obtain the following geometric reformulation of the above theorem.
\begin{cor}\label{gdegen}
There exists a flat family $\mathcal X\subset\bP(\bC^N)\times\mathbb A^1$ over $\mathbb A^1$ such that for the projection $\pi$ onto $\mathbb A^1$ any general fibre $\pi^{-1}(a)$ with $a\neq 0$ is isomorphic to $X$ while the special fibre $\pi^{-1}(0)$ is isomorphic to $X^w$.
\end{cor}
This corollary states that $X^w$ is a flat degeneration of $X$, a flat degeneration of this form is known as a \emph{Gr\"obner degeneration}.

We now move on to define the Gr\"obner fan of $I$ which parametrizes its initial ideals. We retain the assumption that $I$ is homogeneous but $I$ need not be radical here. For an ideal $J\subset S$ denote $C(I,J)\subset\bR^N$ the set of points $w$ for which $\initial_w I=J$. The nonempty sets $C(I,J)$ form a partition of $\bR^N$ with $w$ contained in $C(I,\initial_w I)$. This partition is known as the \textbf{Gr\"obner fan} of $I$ (introduced in~\cite{MR}), its basic properties are summed up in the below theorem. This information can be found in Chapters 1 and 2 of~\cite{S} (see also~\cite[Theorem 1.1]{M}). 
\begin{theorem}\label{gfans}
\hfill
\begin{enumerate}[label=(\alph*)]
\item The set $C(I,J)$ is nonempty for only finitely many ideals $J\subset S$.
\item Every nonempty $C(I,J)$ is a relatively open polyhedral cone.
\item Together all the nonempty $C(I,J)$ form a polyhedral fan with support $\bR^N$. This means that every face of the closure $\overline{C(I,J)}$ is itself the closure of some $C(I,J')$. 
\item\label{IJ'J} For nonempty $C(I,J)$ and $C(I,J')$ the cone $\overline{C(I,J)}$ is a face of $\overline{C(I,J')}$ if and only if $J'$ is an initial ideal of $J$ (i.e.\ $C(J,J')$ is nonempty).
\item A nonempty cone $C(I,J)$ is maximal in the Gr\"obner fan if and only if $J$ is monomial.
\end{enumerate}
\end{theorem}

\subsection{Posets and distributive lattices}

Consider a finite poset $(P,<)$. Let $\cJ(P,<)$ denote the set of order ideals in $(P,<)$ (we use the term ``order ideal'' synonymously to ``lower set''). It is easy to see that $\cJ(P,<)$ can be viewed as a distributive lattice with union as join, intersection as meet and inclusion as the order relation. A classical result due to Garret Birkhoff known as \textbf{Birkhoff's representation theorem} or the \emph{fundamental theorem of finite distributive lattices} states the following.
\begin{theorem}\label{birkhoff}
For every finite distributive lattice $(\cL,\lor,\land)$ there exists a unique up to isomorphism poset $(P,<)$ such that $(\cL,\lor,\land)$ and $(\cJ(P,<),\cup,\cap)$ are isomorphic.
\end{theorem}
In fact, Birkhoff's theorem also provides an explicit realization of $(P,<)$ as the subposet of join-irreducible elements in $\cL$ but we will not be using this notion here, see~\cite[Theorem 9.1.7]{HH}.

A small remark on notation: when speaking of distributive lattices (which appear in this paper as lattices of order ideals) we will often use the underlying set to denote the whole lattice, e.g.\ write $\cL$ instead of $(\cL,\lor,\land)$. The lattice structure will be clear from the context. 


We say that an order relation $\prec $ on the set $P$ is \textbf{stronger} than $<$ ($<$ is \textbf{weaker} than $\prec$) if $p<q$ implies $p\prec q$. Note that for such a $\prec $ any order ideal in $\cJ(P,\prec )$ lies in $\cJ(P,<)$, so the former set is naturally embedded into the latter. Obviously, this is an embedding of lattices. Recall that the height of a distributive lattice is the length of any maximal chain. One easily checks that the height of $\cJ(P,<)$ is $|P|+1$.
\begin{proposition}\label{sublattices}
For every sublattice $\cL\subset\cJ(P,<)$ of height $|P|+1$ there exists a unique order $\prec $ on $P$ that is stronger than $<$ for which $\cJ(P,\prec )=\cL$.
\end{proposition}
\begin{proof}
Since the height of $\cL$ is $|P|+1$, it contains an increasing chain $J_0\subset\dots\subset J_{|P|}$ of order ideals. This means that $|J_i|=i$ and every $p\in P$ appears as $|J_i|\backslash|J_{i-1}|$ for precisely one $i$. This implies that when $p\neq q$ we can always find $J\in\cL$ that contains exactly one of $p$ and $q$. For $p\neq q$ set $p\prec q$ if and only if every $J\in\cL$ with $q\in J$ also contains $p$, we see that $\prec$ is antisymmetric and thus an order relation.

Obviously, $\prec $ is stronger than $<$. It is also immediate that any $J\in\cL$ lies in $\cJ(P,\prec )$. To show that all order ideals in $(P,\prec )$ lie in $\cL$ it suffices to show that the principal order ideals do, since every order ideal is a union of principal ones. However, the principal order ideal in $(P,\prec )$ generated by $p$ is precisely the intersection of all $J\in\cL$ that contain $p$, hence it lies in $\cL$.

One also sees that if orders $\prec $ and $\prec '$ on $P$ are different, then $\cJ(P,\prec )\neq\cJ(P,\prec ')$ which implies the uniqueness of $\prec $ and completes the proof.\qedhere

\end{proof}

\begin{remark}\label{categorical}
In category-theoretic terms one can view $\cJ$ as a contravariant functor from the category of finite posets to the category of finite distributive lattices. Birkhoff's representation theorem then shows that this functor is a duality of categories. The above proposition essentially asserts that the monomorphism $\cL\hookrightarrow\cJ(P,<)$ corresponds to a bijective epimorphism under this duality.
\end{remark}

Proposition~\ref{sublattices} provides a one-to-one correspondence between sublattices of maximal height in $\cJ(P,<)$ and orders on $P$ stronger than $<$. Now note that a linearly ordered sublattice of maximal height in $\cJ(P,<)$ is precisely a maximal chain in $\cJ(P,<)$ and that $\cJ(P,\prec )$ is a linearly ordered lattice if and only if $\prec $ is a linear order itself. This shows that as a special case we obtain a one-to-one correspondence between maximal chains in $\cJ(P,<)$ and linearizations of $<$ (linear orders that are stronger than $<$).
\begin{cor}
For every maximal chain $C\subset\cJ(P,<)$ there exists a unique linearization $\prec$ of $<$ such that $C=\cJ(P,\prec)$. 
\end{cor}
For a linear order $\prec$ on $P$ we will use the shorthand $C_\prec=\cJ(P,\prec)$.

\subsection{Hibi varieties and order polytopes}

As before, we consider a finite poset $(P,<)$. For every $J\in\cJ(P,<)$ introduce a variable $X_J$ and let $\bC[\cJ(P,<)]$ be the polynomial ring in these variables. The \textbf{Hibi ideal} of $(P,<)$ is the ideal $I^h_{P,<}\subset\bC[\cJ(P,<)]$ generated by the elements 
\begin{equation}\label{hibibinomial}
d(J_1,J_2)=X_{J_1}X_{J_2}-X_{J_1\cup J_2}X_{J_1\cap J_2} 
\end{equation}
for all $J_1,J_2\in\cJ(P,<)$. Note that $d(J_1,J_2)\neq 0$ only if $J_1$ and $J_2$ are $\subset$-incomparable, i.e.\ neither contains the other. The notion of Hibi ideals originates from~\cite{H}, it is more standard to associate the ideal with a distributive lattice rather than a poset but, in view of Birkhoff's theorem, these are two equivalent approaches.

We have the following alternative characterization of the Hibi ideal. Consider variables $z_p$ indexed by $p\in P$ and let $\bC[P,t]$ denote the polynomial ring in the $z_p$ and an additional variable $t$. Define a homomorphism $\varphi:\bC[\cJ(P,<)]\to \bC[P,t]$ by \[\varphi(X_J)=t\prod_{p\in J}z_p.\] It is immediate that the binomials $d(J_1,J_2)$ lie in the kernel of $\varphi$. Moreover, the following holds.
\begin{proposition}[{\cite[Section 2]{H}}]\label{kernel}
$I^h_{P,<}$ is the kernel of $\varphi$.
\end{proposition}

The corresponding \textbf{Hibi variety} is the subvariety $H(P,<)\subset\bP(\bC^{\cJ(P,<)})$ cut out by $I^h_{P,<}$. It turns out that this variety is the toric variety of the order polytope of $(P,<)$ which we now define.

The \textbf{order polytope} $\cO(P,<)$ is a convex polytope in the space $\bR^P$ consisting of points $x$ with coordinates $(x_p, p\in P)$ satisfying $0\le x_p\le 1$ for all $p$ and $x_p\le x_q$ whenever $p>q$. It is immediate from the definition that the set of integer points in $\cO(P,<)$ is the set of indicator functions $\bm 1_J$ of order ideals $J\subset P$. It is also easily seen that each of these integer points is a vertex of $\cO(P,<)$ (see~\cite[Corollary 1.3]{stan}). In particular, the vertices (and integer points) of $\cO(P,<)$ are in natural bijection with $\mathcal J(P,<)$.

It should be pointed out that the original definition in~\cite{stan} contains the reverse (and, perhaps, more natural) inequality $x_p\le x_q$ whenever $p<q$, this amounts to reflecting $\cO(P,<)$ in the point $(\frac12,\dots,\frac12)$. We use the above definition to adhere to the standard convention of considering order ideals in Birkhoff's theorem (rather than order filters).

The order polytope $\cO(P,<)$ has the important property of being \textbf{normal} (see, for instance,~\cite[Theorem 2.5]{FF}). This means that for any integer $k>0$ every integer point in its dilation $k\cO(P,<)$ can be expressed as the sum of $k$ (not necessarily distinct) integer points in $\cO(P,<)$, a property sometimes also known as being ``integrally closed''. In other words, this means that the set $k\cO(P,<)\cap\bZ^P$ is the $k$-fold Minkowski sum of $\cO(P,<)\cap\bZ^P$ with itself.

The toric variety of a normal polytope is easy to describe, its homogeneous coordinate ring is the semigroup ring of the semigroup generated by the polytope's integer points, see~\cite[\textsection2.3]{CLS}. In the case of $\cO(P,<)$ this ring is precisely the image of the map $\varphi$ considered in Proposition~\ref{kernel}. This leads to the following fact which is also sometimes attributed to~\cite{H}.
\begin{proposition}\label{toricvariety}
$H(P,<)$ is isomorphic to the toric variety of the polytope $\cO(P,<)$.
\end{proposition}

\subsection{Hibi-Li varieties and chain polytopes}

Stanley~(\cite{stan}) also defines another polytope associated with a finite poset $(P,<)$, its \textbf{chain polytope} $\cC(P,<)$. It is the set of all $x\in\bR^P$ such that all $x_p\ge 0$ and for any chain $p_1<\dots<p_k$ we have $x_{p_1}+\dots+x_{p_k}\le 1$.

It is easily seen (and shown in~\cite{stan}) that the set of integer points in $\cC(P,<)$ consists of indicator functions $\mathbf 1_A$ for antichains $A$ in $(P,<)$. Moreover, all of these points are vertices of the polytope. Recall that there is a bijection between the set of order ideals and the set of antichains in $(P,<)$ which maps $J\in\cJ(P,<)$ to the antichain $\max_< J$ of its maximal elements. We see that the integer points in $\cC(P,<)$ are in a natural bijection with the integer points of $\cO(P,<)$. Moreover, the following holds.
\begin{lemma}\label{ehrhart}
The polytopes $\cO(P,<)$ and $\cC(P,<)$ are Ehrhart-equivalent: for any $m\in\bZ_{\ge 0}$ the polytopes $m\cO(P,<)$ and $m\cC(P,<)$ have the same number of integer points. This number is equal to the number of weakly increasing tuples $J_1\subseteq\dots\subseteq J_m$ of order ideals in $(P,<)$.
\end{lemma}
\begin{proof}
The first claim is proved in \cite[Theorem 4.1]{stan}. Also, for a weakly increasing tuple $J_1\subseteq\dots\subseteq J_m$ of order ideals the integer point $\mathbf 1_{J_1}+\dots+\mathbf 1_{J_m}$ lies in $m\cO(P,<)$ and for an integer point $x\in m\cO(P,<)$ the order ideals $J_i=\{p|x_p\ge m-i+1\}$ form a weakly increasing tuple. This proves the second claim.
\end{proof}

Similarly to the Hibi variety, the toric variety of $\cC(P,<)$ can be realized as the zero set of a certain ideal in $\bC[\cJ(P,<)]$. We will call this variety the \textbf{Hibi-Li variety} since it was studied in~\cite{HL}.

Following the latter paper we introduce another binary operation $*_{P,<}$ on $\cJ(P,<)$. 
\begin{definition}
For $J_1,J_2\in\cJ(P,<)$ denote by $J_1 *_{P,<} J_2$ the ideal generated by the antichain \begin{equation}\label{star}
(J_1\cap J_2)\cap (\max\nolimits_< J_1\cup\max\nolimits_< J_2)=\max\nolimits_<(J_1\cap J_2)\cap (\max\nolimits_< J_1\cup\max\nolimits_< J_2).
\end{equation}
\end{definition}
The key property of this operation (which can be viewed as an alternative definition) is the equality 
\begin{equation}\label{stardef}
\mathbf 1_{\max_< J_1}+\mathbf 1_{\max_< J_2}=\mathbf 1_{\max_< (J_1\cup J_2)}+\mathbf 1_{\max_< (J_1*_{P,<} J_2)} 
\end{equation}
which is easily verified. The \textbf{Hibi-Li ideal} is then the ideal $I^{hl}_{P,<}\subset\bC[\cJ(P,<)]$ generated by the binomials \[X_{J_1}X_{J_2}-X_{J_1\cup J_2}X_{J_1*_{P,<}J_2}.\]

Similarly to the order polytope case, we have the following.
\begin{proposition}[\cite{HL}, Section 2]
The ideal $I^{hl}_{P,<}$ is the kernel of the map from $\bC[\cJ(P,<)]$ to $\bC[P,t]$ taking $X_J$ to $t\prod_{p\in\max_< J}z_p$. The subvariety cut out by $I^{hl}_{P,<}$ in $\bP(\bC^{\cJ(P,<)})$ is the toric variety of $\cC(P,<)$.
\end{proposition}

\subsection{The monomial ideal and triangulations}\label{monomialideal}

One can also associate a monomial ideal with the poset $(P,<)$ (or lattice $\cJ(P,<)$). This is the ideal $I^m_{P,<}\subset\bC[\cJ(P,<)]$ generated by the monomials $X_{J_1}X_{J_2}$ with $J_1$ and $J_2$ incomparable with respect to $\subset$. A simple property of this ideal is as follows (see~\cite[Proposition 5.3]{M} or~\cite{H,HL} where the stronger ASL property is proved).
\begin{proposition}\label{monomialinitial}
$I^m_{P,<}$ is an initial ideal of both $I^h_{P,<}$ and $I^{hl}_{P,<}$.
\end{proposition}

Note that $I^m_{P,<}$ is spanned by monomials $X_{J_1}\dots X_{J_M}$ such that there is no maximal chain in $\cJ(P,<)$ containing all $J_i$. For a maximal chain $C\subset\cJ(P,<)$ let $I_C$ be the prime ideal generated by all $X_J$ with $J\notin C$. We see that $I^m_{P,<}$ is the intersection of all $I_C$. Let $M(P,<)\subset\bP(\bC^{\cJ(P,<)})$ denote the zero set of $I^m_{P,<}$, it is a Gr\"obner degeneration of both the Hibi and the Hibi-Li varieties in view of Proposition~\ref{monomialinitial}. We obtain
\begin{proposition}
The irreducible components of $M(P,<)$ are the subspaces $\bP(\bC^C)$ for all maximal chains $C\subset\cJ(P,<)$.
\end{proposition}

An observation due to~\cite{FL} is that this variety corresponds to the canonical triangulation of the order polytope. Indeed, consider a linearization $\prec$ of $<$. Then the order polytope $\cO(P,\prec)$ is a simplex which is contained in $\cO(P,<)$ with vertices $\one_J$, $J\in C_\prec$. Such simplices for all $\prec$ are seen to form a triangulation of $\cO(P,<)$ (see~\cite[Section 5]{stan}). 

Since the integer points of a simplex $\cO(P,\prec)$ are of the form $\one_J$ with $J\in\cJ(P,<)$, its toric variety can be naturally embedded into $\bP(\bC^{\cJ(P,<)})$ as zero set of the ideal generated by $X_J$ with $\one_J\notin \cO(P,\prec)$. However, this is precisely the ideal $I_{C_\prec}$ and the obtained subvarieties in $\bP(\bC^{\cJ(P,<)})$ are precisely the irreducible components of $M(P,<)$.

A similar observation can be made concerning the chain polytope. Consider the map $\phi_<$ from $\bR^P$ to itself with \[\phi_<(x)_p=x_p-\max_{q>p} x_q.\] By~\cite[Theorem 3.2]{stan} $\phi_<$ provides a continuous piecewise linear bijection from $\cO(P,<)$ to $\cC(P,<)$. One also sees that this map is linear on every simplex $\cO(P,\prec)$ and, therefore, $\phi_<(\cO(P,\prec))$ is also a simplex which we denote $\Delta_{\prec,<}$. Evidently, these simplices form a triangulation of $\cC(P,<)$. Note that, since $\phi_<(\one_J)=\one_{\max_> J}$ for $J\in\cJ(P,<)$, the vertices of $\Delta_{\prec,<}$ are the points $\one_{\max_< J}$ with $J\in\cJ(P,\prec)=C_\prec$.

The obtained triangulation of $\cC(P,<)$ can be said to correspond to $M(P,<)$ in the same way as the triangulation of $\cO(P,<)$. We will formalize and generalize this correspondence between subdivisions and semitoric varieties in the next two subsections.

\subsection{Regular subdivisions and secondary fans}\label{secondary}

The notions and results here are due to Gelfand, Kapranov and Zelevinsky, see~\cite[Chapter 7]{GKZ}. They, however, use the terms ``coherent subdivision'' and ``coherent triangulation'' rather than ``regular subdivision'' and ``regular triangulation'' that are commonly used today.

Consider a convex polytope $Q\subset\bR^n$ of dimension $n$ with set of vertices $V=\{v_1,\dots,v_k\}$ and a point $c=(c_1,\dots,c_k)\in\bR^V$. Let $S\subset \bR^n\times\bR$ be the union of rays $v_i\times\{x\le c_i\}$ for $i\in[1,k]$ and let $T$ be the convex hull of $S$. It is evident that $T$ is an $(n+1)$-dimensional convex polyhedron with $\rho(T)=Q$ where $\rho$ denotes the projection onto $\bR^n$. Furthermore, $T$ has two kinds of facets $F$: those with $\dim\rho(F)=n$ and those with $\dim\rho(F)=n-1$. Each of the former facets is bounded and is the convex hull of some subset of the points $v_k\times c_k$. Each of the latter facets is an unbounded convex hull of a union of rays in $S$. One could say that the bounded facets are the ones you see when you ``look at $T$ from above''. 

Let $F_1,\dots,F_m$ be the bounded facets of $T$ and denote $Q_i=\rho(F_i)$. The set $\{Q_1,\dots,Q_m\}$ is a polyhedral subdivision of $Q$ which we denote $\Theta_Q(c)$. This means that $\bigcup Q_i=Q$, all $\dim Q_i=n$ and the $Q_i$ together with all their faces form a polyhedral complex. A polyhedral subdivision of the form $\Theta_Q(c)$ is known as a \textbf{regular subdivision}. 

Note that the union of the bounded facets $F_i$ is the graph of a continuous convex piecewise linear function on $Q$ and that the $Q_i$ are the (maximal) domains of linearity of this function. In fact, this function is the pointwise minimum of all convex functions $f$ on $Q$ with $f(v_i)=c_i$.

Regular subdivisions of $Q$ are also parametrized by a polyhedral fan. For a regular subdivision $\Theta$ denote $C(\Theta)\subset\bR^V$ the set of $c$ such that $\Theta_Q(c)=\Theta$.
\begin{theorem}\label{secondaryfan}
\hfill
\begin{enumerate}[label=(\alph*)]
\item Every $C(\Theta)$ is a relatively open polyhedral cone, together these cones form a polyhedral fan with support $\bR^V$ known as the \textbf{secondary fan} of $Q$. 
\item $\overline{C(\Theta')}$ is a face of $\overline{C(\Theta)}$ if and only if $\Theta$ is a refinement of $\Theta'$.
\item $C(\Theta)$ is a maximal cone in the secondary fan if and only if $\Theta$ is a triangulation of $Q$ (such triangulations are known as \textbf{regular triangulations}).
\item $C(\{Q\})$ is the minimal cone in the secondary fan, it is a linear space of dimension $n+1$.
\end{enumerate}
\end{theorem}

\subsection{Zhu's theorem}

For a finite set of integer points $A\subset\bZ^n$ denote $S_A=\bC[\{X_a\}_{a\in A}]$. We have the corresponding toric ideal $I_A\subset S_A$ which is the kernel of the map from $S_A$ to $\bC[z_1^{\pm1},\dots,z_n^{\pm1},t]$ taking $X_a$ to $tz^a$. For a subset $B\subset A$ let $I_{B,A}\subset S_A$ denote the ideal generated by all $X_a$ with $a\notin B$ and by $I_B\subset S_B$ (note that $S_B\subset S_A$). Evidently, the zero set of $I_{B,A}$ in $\bP(\bC^A)$ coincides with the zero set of $I_B$ in $\bP(\bC^B)\subset\bP(\bC^A)$. Let us write $V(Q)$ to denote the vertex set of a lattice polytope $Q$, then Zhu's theorem is as follows.
\begin{theorem}[\cite{Zh}, Theorem 3]\label{zhu}
For a lattice polytope $Q\subset\bR^n$ and $w\in\bR^{V(Q)}$ let $Q_1,\dots,Q_m$ be the parts of the subdivision $\Theta_Q(w)$. Then, in the ring $S_{V(Q)}$, the ideal $\bigcap_{i=1}^m I_{V(Q_i),V(Q)}$ is the radical of the ideal $\initial_w I_{V(Q)}$.
\end{theorem}

A simple consequence of this result is as follows.
\begin{cor}\label{conescoincide0}
In the setting of the Theorem~\ref{zhu}, if $\initial_w I_{V(Q)}$ is radical, then the cones $C(I_{V(Q)},\initial_w I_{V(Q)})$ and $C(\Theta_Q(w))$ in $\bR^{V(Q)}$ coincide.
\end{cor}
\begin{proof}
Again, let $Q_1,\dots,Q_m$ be the parts of the subdivision $\Theta_Q(w)$ and denote $I=\bigcap_{i=1}^m I_{V(Q_i),V(Q)}$. Zhu's theorem shows that $I=\initial_w I_{V(Q)}$ and that $C(\Theta_Q(w))$ lies in the union of all $C(I_{V(Q)},I')$ with $\sqrt{I'}=I$. However, for an initial ideal $I'$ of $I_{V(Q)}$, if $I'\neq I$ and $\sqrt{I'}=I=\initial_w I_{V(Q)}$, then $I'\subsetneq I$, which is impossible since both have the same Hilbert series as $I_{V(Q)}$ (we always consider Hilbert series with respect to the grading by degree). We see that $C(\Theta_Q(w))\subset C(I_{V(Q)},I)$, the rest follows from~\cite[Theorem 8.15]{S} by which the Gr\"obner fan of $I_{V(Q)}$ refines the secondary fan of $Q$ (this can, in fact, also be easily derived from Zhu's theorem).
\end{proof}

We immediately have the following application which we will generalize later (see Propositions~\ref{groebnercone} and~\ref{secondarycone}). Here recall that the vertex sets of each of $\cO(P,<)$ and $\cC(P,<)$ are in bijection with $\cJ(P,<)$ which allows us to identify both $\bR^{V(\cO(P,<))}$ and $\bR^{V(\cC(P,<))}$ with $\bR^{\cJ(P,<)}$ and both of $S_{V(\cO(P,<))}$ and $S_{V(\cC(P,<))}$ with $\bC[\cJ(P,<)]$.
\begin{proposition}\label{conescoincide1}
For a finite poset $(P,<)$ we have the following equalities of cones in $\bR^{\cJ(P,<)}$: \[C(I^h_{P,<},I^m_{P,<})=C(\{\cO(P,\prec),\prec\text{ is a linearization of }<\})\] and \[C(I^{hl}_{P,<},I^m_{P,<})=C(\{\Delta_{\prec,<},\prec\text{ is a linearization of }<\}).\] In particular, the triangulations in the right-hand sides are regular.
\end{proposition}
\begin{proof}
Note that $I_{V(\cO(P,<))}=I^h_{P,<}$. Since the ideal $I^m_{P,<}$ is radical, the cone $C(I^h_{P,<},I^m_{P,<})$ must coincide with some cone $C(\{Q_1,\dots,Q_m\})$ in the secondary fan of $\cO(P,<)$ by Corollary~\ref{conescoincide0}. By Theorem~\ref{zhu} the (necessarily prime) ideals $I_{V(Q_i),V(\cO(P,<))}$ are the prime components of $I^m_{P,<}$. However, we have seen that the prime components of $I^m_{P,<}$ are the ideals $I_C$ with $C$ ranging over all maximal chains in $\cJ(P,<)$. For a linearization $\prec$ of $<$ the only lattice polytope $Q\subset \cO(P,<)$ with $I_{V(Q),V(\cO(P,<))}=I_{C_\prec}$ is $\cO(P,\prec)$. This shows that $\{Q_1,\dots,Q_m\}$ is precisely the triangulation into the simplices $\cO(P,\prec)$. The second identity is proved similarly.
\end{proof}

\subsection{Multiprojective varieties}

Consider positive integers $n$ and $d_1,\dots,d_n$, denote \[\bP=\bP(\bC^{d_1})\times\dots\times\bP(\bC^{d_n}).\] Consider variables $X_i^j$ with $j\in[1,n]$ and $i\in[1,d_j]$ and denote $S=\bC[\{X_i^j\}]$. The ring $S$ is the \textbf{multihomogeneous coordinate ring} of $\bP$, zero sets of multihomogeneous ideals in $S$ (ideals separately homogeneous in each group of variables $X_1^j,\dots,X_{d_j}^j$) are subvarieties in $\bP$. 

We will use the following ``multiprojective Nullstellensatz'' characterizing the vanishing ideals of subvarieties in $\bP$. Recall that for ideals $I,J\subset S$ one says that $I$ is saturated with respect to $J$ if $(I:J)=I$ where $(I:J)$ is the ideal $\{s\in S| sJ\subset I\}$. For a multihomogeneous ideal $I\subset S$ let $\mathcal V(I)\subset\bP$ denote its zero set.
\begin{theorem}\label{hilbert}
A multihomogeneous ideal $I\subset S$ is the vanishing ideal of $\mathcal V(I)$ if and only if $I$ is radical and is saturated with respect to each of the ``irrelevant'' ideals $\mathfrak M_j=\langle X_1^j,\dots,X_{d_j}^j\rangle$ (or, equivalently, with respect to $\mathfrak M_1\dots\mathfrak M_n$).
\end{theorem}
\begin{proof}
Let $V_{\mathrm{aff}}(I)$ denote the zero set of $I$ in the affine space $\mathbb A=\bigoplus_j \bC^{d_j}$. Denote $\mathbb A_j=V_{\mathrm{aff}}(\mathfrak M_j)$, this is a coordinate subspace. For a multihomogeneous ideal $J\subset S$ we have $\mathcal V(I)=\mathcal V(J)$ if and only if \[V_{\mathrm{aff}}(I)\backslash(\mathbb A_1\cup\dots\cup \mathbb A_n)=V_{\mathrm{aff}}(J)\backslash(\mathbb A_1\cup\dots\cup\mathbb A_n).\] Therefore, $I$ is the vanishing ideal of $\mathcal V(I)$ if and only if $V_{\mathrm{aff}}(I)$ is precisely the closure $\overline{V_{\mathrm{aff}}(I)\backslash(\mathbb A_1\cup\dots\cup \mathbb A_n)}$. The claim now follows from the classical Nullstellensatz and the fact that the vanishing ideal of the difference of two affine subvarieties is the saturation of their vanishing ideals.
\end{proof}

\begin{remark}
Somewhat surprisingly, the authors were unable to find this theorem in the literature, see~\cite{MO}.
\end{remark}

We proceed to prove a useful fact concerning multiprojective realizations of toric varieties. Consider $n$ lattice polytopes $P_1,\dots,P_n\subset\bR^m$. We assume that the Minkowski sum $P=P_1+\dots+P_n$ is normal and, furthermore, that 
\begin{equation}\label{minkowski}
P\cap\bZ^m=P_1\cap\bZ^m+\dots+P_n\cap\bZ^m,
\end{equation}
i.e.\ every integer point in $P$ decomposes into a sum of integer points in the $P_j$.

Denote the integer points $P_j\cap\bZ^m=\{a_1^j,\dots,a_{d_j}^j\}$. We again set $S=\bC[\{X_i^j\}]$ where $j\in[1,n]$, $i\in[1,d_j]$. Denote $R=\bC[z_1^{\pm1},\dots,z_m^{\pm1},t_1,\dots,t_n]$. We define a homomorphism $\varphi:S\to R$ by $\varphi(X^j_i)=t_j z^{a^j_i}$ and denote $I=\ker\varphi$. We also define the product $\bP$ as above and again view $S$ as its multihomogeneous coordinate ring.
\begin{lemma}\label{multiproj}
$I$ is multihomogeneous and $\mathcal V(I)$ is isomorphic to the toric variety of $P$. Moreover, $I$ is the vanishing ideal of $\mathcal V(I)$.
\end{lemma}
\begin{proof}
The multihomogeneity is immediate from the definition. Let $P\cap\bZ^m=\{b_1,\dots,b_d\}$. The toric variety of $P$ can be realized in $\bP(\bC^d)$ as the closure $X_P$ of the set of points of the form 
\begin{equation}\label{X_P}
(x^{b_1}:\dots:x^{b_d})
\end{equation}
where $x\in(\bC^*)^m$. Meanwhile, $\mathcal V(I)$ is the closure of the set of points of the form 
\begin{equation}\label{V(I)}
(x^{a_1^1}:\dots:x^{a_{d_1}^1})\times\dots\times(x^{a_1^m}:\dots:x^{a_{d_m}^m})\in\bP
\end{equation}
where $x\in (\bC^*)^m$, this follows from the definition of $\varphi$. 

Now consider the space $\mathbb U=\bP(\bC^{d_1\dots d_m})$, its homogeneous coordinates are enumerated by tuples $(i_1,\dots,i_m)$ with $i_k\in [1,d_k]$. We have an embedding $\iota:\bP(\bC^d)\hookrightarrow\mathbb U$ given in homogeneous coordinates by \[\iota(y_1:\dots:y_d)_{i_1,\dots,i_m}=y_j\] where $a_{i_1}^1+\dots+a_{i_m}^m=b_j$. This map is injective due to~\eqref{minkowski}. We also have the Segre embedding $\sigma:\bP\hookrightarrow\mathbb U$. The image of~\eqref{X_P} under $\iota$ coincides with the image of~\eqref{V(I)} under $\sigma$, i.e.\ $\iota(X_P)=\sigma(\mathcal V(I))$.

To verify the last claim we use Theorem~\ref{hilbert}. $I$ is radical since it is prime. It is also evident that for $s\in S$ and any $X^j_i$ we have $s X^j_i\in I$ if and only if $s\in I$ which implies saturatedness with respect to the $\mathfrak M_j$.
\end{proof}

\section{Relative poset polytopes}\label{nonmarked}

\subsection{Combinatorial properties}

Choose two finite posets $(P,<)$ and $(P,<')$ with the same underlying set such that $<'$ is weaker than $<$. Recall that the lattice $\cJ=\cJ(P,<)$ is naturally embedded into $\cJ'=\cJ(P,<')$ as a sublattice. We will assume that the sublattice $\cJ\subset\cJ'$ is closed under $*'=*_{P,<'}$: for $J_1,J_2\in\cJ$ we have $J_1*'J_2\in\cJ$.

We can now give the main definition. 
\begin{definition}
The \textit{relative poset polytope} $\cR(P,<,<')$ is the convex hull of the points $\mathbf 1_{\max_{<'}J}\in\bR^P$ for all $J\in\cJ$. 
\end{definition}
We immediately note that all of these points are pairwise distinct and are vertices of the defined polytope.
\begin{proposition}\label{rppverts}
$\cR(P,<,<')$ has $|\cJ|$ vertices: the points $\mathbf 1_{\max_{<'}J}\in\bR^P$,~${J\in\cJ}$.    
\end{proposition}

In other words, $\cR(P,<,<')$ is the convex hull of those vertices of $\cC(P,<')$ which correspond to ideals in $(P,<)$. We use the word ``relative'' to imply that it is a poset polytope for the poset $(P,<)$ defined in relation to the weaker order $<'$. Obviously, the definition makes sense regardless of $\cJ$ being closed under $*'$ but this assumption is necessary for $\cR(P,<,<')$ to have the nice properties proved here.

\begin{example}
If $<'=<$, then $\cJ$ is indeed closed under $*'=*_{P,<}$ and the vertices of $\cR(P,<,<')$ are the points $\mathbf 1_{\max_< J}$ with $J\in\cJ$, i.e.\ it is the chain polytope $\cC(P,<)$. On the other extreme, let $<'$ be the trivial order relation on $P$ for which no two elements are comparable. Then for any $J\in\cJ$ we have $\max_{<'}J=J$, hence $*'=\cap$ and $\cR(P,<,<')$ is the order polytope $\cO(P,<)$. Among all $<'$ weaker than $<$, the trivial one is the weakest while $<$ is the strongest, hence the $\cR(P,<,<')$ can be said to interpolate between $\cO(P,<)$ and $\cC(P,<)$. We also note that if $\prec$ is a linearization of $<'$, then $\cR(P,\prec,<')$ is precisely the simplex $\Delta_{\prec,<'}$ considered in Section~\ref{monomialideal}.
\end{example}

Below is, perhaps, the key property of $\cR(P,<,<')$ and the reason for us to require $\cJ$ to be closed under $*'$.
\begin{lemma}\label{triangulation}
$\cR(P,<,<')$ is the union of those simplices $\Delta_{\prec,<'}$ for which $\prec$ is a linearization of $<$.
\end{lemma}
\begin{proof}
For any linearization $\prec$ of $<$ any order ideal $J$ in $(P,\prec)$ is also an ideal in $(P,<)$. Hence, the vertex $\mathbf 1_{\max_{<'} J}$ of $\Delta_{\prec,<'}$ is also a vertex of $\cR(P,<,<')$ and the latter contains the former.

We are to show that any point $x\in \cR(P,<,<')$ is contained in one of the $\Delta_{\prec,<'}$ with $\prec$ a linearization of $<$. Indeed, $x$ can be expressed as convex linear combination of vertices of $\cR(P,<,<')$: \[x=\sum_{J\in\cJ} c_J\mathbf 1_{\max_{<'}J}\] where all $c_J\ge 0$ and $\sum c_J=1$. We will assume that all $c_J$ are rational, it suffices to prove the claim for such $x$ since they form a dense subset. 

We prove that the $c_J$ can be chosen in such a way that if $c_{J_1},c_{J_2}>0$, then $J_1$ and $J_2$ are comparable in terms of the poset structure on $\cJ$, i.e.\ one contains the other. Suppose that $J_1$ and $J_2$ are incomparable and $c_{J_1},c_{J_2}>0$. Let us describe a procedure of obtaining a new set of coefficients $\tilde c_J$ such that 
\begin{equation}\label{step}
x=\sum_{J\in\cJ} \tilde c_J\mathbf 1_{\max_{<'}J}.
\end{equation}
We set $\tilde c_J=c_J$ unless  $J\in\{J_1,J_2,J_1\cup J_2,J_1*'J_2\}$ and 
\begin{align*}
\tilde c_{J_1}&=c_{J_1}-\min(c_{J_1},c_{J_2}),\\
\tilde c_{J_2}&=c_{J_2}-\min(c_{J_1},c_{J_2}),\\
\tilde c_{J_1\cup J_2}&=c_{J_1\cup J_2}+\min(c_{J_1},c_{J_2}),\\
\tilde c_{J_1*' J_2}&=c_{J_1*' J_2}+\min(c_{J_1},c_{J_2}).
\end{align*}
The equality~\eqref{step} follows directly from the definition of $*_{P,<'}$. Now, since $|J_1|+|J_2|=|J_1\cup J_2|+|J_1\cap J_2|$ and $J_1*' J_2\subset J_1\cap J_2$, we have 
\begin{multline}\label{squaresineq}
|P\backslash J_1|^2+|P\backslash J_2|^2\le |P\backslash (J_1\cup J_2)|^2+|P\backslash (J_1\cap J_2)|^2\le\\ |P\backslash (J_1\cup J_2)|^2+|P\backslash (J_1*'J_2)|^2. 
\end{multline}
We obtain $\sum c_J|P\backslash J|^2<\sum \tilde c_J|P\backslash J|^2$. Since the least common denominator of the $\tilde c_J$ is no greater than that of the $c_J$, we see that we may not repeat this procedure infinitely. Hence, by iterating the procedure we will arrive at a convex linear combination $x=\sum b_J\mathbf 1_{\max_{<'}J}$ in which those $J$ for which $b_J>0$ are pairwise comparable as desired.

Thus, all $J$ with $b_J>0$ are contained in some maximal chain $C$ in the lattice $\cJ$. However, $C=C_\prec$ for some linearization $\prec$ of $<$ and the points $\mathbf 1_{\max_{<'}J}$ with $J\in C$ are precisely the vertices of $\Delta_{\prec,<'}$. We obtain $x\in\Delta_{\prec,<'}$.
\end{proof}

\begin{remark}
One may, in fact, show that for any order $<'$ weaker than $<$ the condition that $\cJ$ is closed under $*'$ is equivalent to the convex hull of $\{\one_{\max_{<'}J}\}_{J\in\cJ}$ being a union of several simplices $\Delta_{\prec,<'}$. It would be interesting to find other equivalent conditions, in particular, in purely order-theoretic (rather than convex geometric) terms.
\end{remark}

\begin{example}
We give an example where both $\cJ$ is not closed under $*'$ and $\cR(P,<,<')$ is not a union of the simplices $\Delta_{\prec,<'}$. Consider $P=\{a,b,c,d\}$ with partial orders given by $a<b$, $b<c$, $b<d$ and $a<'c$, $a<' d$. Then $J_1=\{a,b,c\}$ and $J_2=\{a,b,d\}$ lie in $\cJ$ and $\cJ'$. On one hand, $J_1*'J_2=\{b\}$ but $\{b\}\notin\cJ$ which means that $\cJ$ is not closed under $*'$. On the other, $<$ has two linearizations: $a\prec_1 b\prec_1 c\prec_1 d$ and $a\prec_2 b\prec_2 d\prec_2 c$. The point $x$ with $x_a=0$, $x_b=1$ and $x_c=x_d=1/2$ (i.e.\ the midpoint between $\one_{\max_{<'}J_1}$ and $\one_{\max_{<'}J_2}$) is seen to lie outside of both $\Delta_{\prec_1,<'}$ and $\Delta_{\prec_2,<'}$.
\end{example}

We have two important consequences.
\begin{cor}
$\cR(P,<,<')$ is a normal polytope.
\end{cor}
\begin{proof}
It is subdivided into normal polytopes (unimodular simplices) $\Delta_{\prec,<'}$.
\end{proof}

\begin{theorem}
$\cR(P,<,<')$ is Ehrhart-equivalent to $\cO(P,<)$ and $\cC(P,<)$.
\end{theorem}
\begin{proof}
For any $m\in\bZ_{\ge 0}$ the simplices $m\Delta_{\prec,<'}$ with $\prec$ a linearization of $<$ form a triangulation of $m\cR(P,<,<')$. Therefore, any integer point $x\in m\cR(P,<,<')$ can be uniquely expressed as a convex linear combination of the vertices of some $m\Delta_{\prec,<'}$. Moreover, since $\Delta_{\prec,<'}$ is a unimodular simplex, the coefficients in this linear combination will lie in $\bZ/m$. In other words, we can uniquely write $x$ in the form \[x=\sum_{J\in C_\prec}c_J\mathbf 1_{\max_{<'}J}\] where $c_J\in\bZ_{\ge0}$ and $\sum c_J=m$. Conversely, any linear combination of this form is obviously an integer point in $m\cR(P,<,<')$. Hence, the integer points in $m\cR(P,<,<')$ are in bijection with the set of size $m$ submultisets of all maximal chains in $\cJ$. This is precisely the set of size $m$ weakly increasing tuples in $\cJ$ and the claim follows from Lemma~\ref{ehrhart}.
\end{proof}

With the above argument we have also shown the following.
\begin{cor}\label{inctuple}
Every integer point $x\in m\cR(P,<,<')$ can be uniquely expressed in the form \[x=\mathbf 1_{\max_{<'}J_1}+\dots+\mathbf 1_{\max_{<'}J_m}\] with $J_1\subset\dots\subset J_m$ a weakly increasing tuple in $\cJ(P,<)$.
\end{cor}

\subsection{Toric variety and toric ideal}

Let $H(P,<,<')$ denote the toric variety of the polytope $\cR(P,<,<')$, we will call it the \textbf{relative Hibi variety}. Since $\cR(P,<,<')$ is normal and its integer points are parametrized by $\cJ$, the variety $H(P,<,<')$ can be embedded into $\bP(\bC^\cJ)$. It is the subvariety cut out by the kernel of the map $\varphi_{P,<,<'}:\bC[\cJ]\to\bC[P,t]$ given by 
\begin{equation}\label{relkernel}
\varphi_{P,<,<'}(X_J)=t\prod_{p\in\max_{<'}J} z_p. 
\end{equation}
We denote this kernel $I_{P,<,<'}$ and call it the \textbf{relative Hibi ideal}.

\begin{proposition}
The monomial ideal $I^m_{P,<}$ is an initial ideal of $I_{P,<,<'}$.
\end{proposition}
\begin{proof}
Since $\cR(P,<,<')$ and $\cO(P,<)$ are Ehrhart-equivalent, the corresponding toric ideals $I^h_{P,<}$ and $I_{P,<,<'}$ have the same Hilbert series. In view of Proposition~\ref{monomialinitial}, $I^m_{P,<}$ has the same Hilbert series as $I^h_{P,<}$. Therefore, it suffices to show that $I^m_{P,<}$ is contained in an initial ideal of $I_{P,<,<'}$.

For any $J_1,J_2\in\cJ$ the binomial
\begin{equation}\label{binomial1}
X_{J_1}X_{J_2}-X_{J_1\cup J_2}X_{J_1*'J_2}
\end{equation}
is easily seen to lie in $I_{P,<,<'}$. Now consider $w\in\bR^\cJ$ with $w_J=|P\backslash J|^2$. If $J_1$ and $J_2$ are incomparable, the first inequality in~\eqref{squaresineq} is strict and we have \[w_{J_1}+w_{J_2}<w_{J_1\cup J_2}+w_{J_1*'J_2}.\] Therefore, the initial part of~\eqref{binomial1} with respect to $w$ is $X_{J_1}X_{J_2}$ and we have $I^m_{P,<}\subset\initial_w I_{P,<,<'}$.
\end{proof}

We deduce
\begin{theorem}\label{quadratic}
$I_{P,<,<'}$ is generated by the binomials 
\begin{equation}\label{binomial}
X_{J_1}X_{J_2}-X_{J_1\cup J_2}X_{J_1*'J_2}
\end{equation}
with $J_1,J_2$ ranging over pairs of incomparable order ideals in $\cJ$.
\end{theorem}
\begin{proof}
The proof of the previous proposition shows that $I^m_{P,<}$ is an initial ideal of the subideal in $I_{P,<,<'}$ generated by these binomials. Hence, this subideal must coincide with all of $I_{P,<,<'}$.
\end{proof}

\begin{proposition}\label{groebnercone}
The maximal cone $C(I_{P,<,<'},I^m_{P,<})$ in the Gr\"obner fan of $I_{P,<,<'}$ consists of those $w\in\bR^\cJ$ which satisfy \begin{equation}\label{definingineq}
w_{J_1}+w_{J_2}<w_{J_1\cup J_2}+w_{J_1*'J_2}
\end{equation}
for any incomparable $J_1$ and $J_2$.
\end{proposition}
\begin{proof}
Evidently, if some $w$ does not satisfy this inequality for $J_1$ and $J_2$, then the corresponding initial part of the binomial~\eqref{binomial} will not lie in $I^m_{P,<}$. Conversely, if the inequality is satisfied, then the initial part of~\eqref{binomial} equals $X_{J_1}X_{J_2}$. Thus, if all the inequalities are satisfied, $\initial_w I_{P,<,<'}$ contains $I^m_{P,<}$ and hence coincides with the latter.
\end{proof}

\subsection{Regular subdivisions and secondary fan}\label{subdivisions}

Next let us consider the secondary fan of $\cR(P,<,<')$. By Proposition~\ref{rppverts} the vertices of $\cR(P,<,<')$ are in bijection with $\cJ$, we may view it as a fan in $R^\cJ$.
\begin{proposition}\label{secondarycone}
The triangulation \[\Psi(P,<,<')=\{\Delta_{\prec,<'}|\prec\text{ is a linearization of }<\}\] of $\cR(P,<,<')$ is regular. The maximal cone $C(\Psi(P,<,<'))$ consists of points $w\in\bR^\cJ$ for which \[w_{J_1}+w_{J_2}< w_{J_1\cup J_2}+w_{J_1*'J_2}\] for all $J_1,J_2\in\cJ$.
\end{proposition}
\begin{proof}
The fact that $\Psi(P,<,<')$ is regular follows directly from Proposition~\ref{conescoincide1}, since it is a subcomplex of a regular triangulation of the larger polytope $\cC(P,<')$. For any linearization $\prec$ of $<$ the ideal $I_{V(\Delta_{\prec,<'}),V(\cR(P,<,<'))}$ (where we identify $S_{V(\cR(P,<,<'))}$ with $\bC[\cJ]$) is precisely the ideal $I_{C_\prec}$ (see Subsection~\ref{monomialideal}). The intersection of these ideals over all linearizations is $I^m_{P,<}$ which is radical. Furthermore, $I_{V(\cR(P,<,<'))}=I_{P,<,<'}$ because $\cR(P,<,<')$ is normal and has no lattice points other than its vertices. Hence, by Corollary~\ref{conescoincide0}, we must have $C(\Psi(P,<,<'))=C(I_{P,<,<'},I^m_{P,<})$, and the proposition now follows from Proposition~\ref{groebnercone}.
\end{proof}

To prove the property of relative poset polytopes that, to a large extent, motivated their definition (Theorem~\ref{partsrpp}) we will need the following lemma.
\begin{lemma}\label{convexunion}
Suppose that linearizations $\prec_1,\dots,\prec_k$ of $<'$ are such that the set $\bigcup_i \Delta_{\prec_i,<'}$ is convex. Then $\bigcup_i \cO(P,\prec_i)$ is also convex.
\end{lemma}
\begin{proof}
Suppose that $\bigcup_i \cO(P,\prec_i)$ is not convex. This means that we have a vertex $\one_J$ of $\bigcup_i \cO(P,\prec_i)$ and a facet $F$ of some $\cO(P,\prec_l)$ such that $F$ is contained in a facet of $\bigcup_i \cO(P,\prec_i)$ and the point $\one_J$ and the simplex $\cO(P,\prec_l)$ lie on opposite sides of the hyperplane containing $F$. Note that the latter hyperplane must be of the form $\{x|x_q=x_r\}$ for some $q,r\in P$ such that $r$ covers $q$ with respect to $\prec_l$. Then all $x\in \cO(P,\prec_l)$ satisfy $x_q\ge x_r$ and we must have $q\notin J$ and $r\in J$.

Recall the map $\phi_{<'}$ defined in Subsection~\ref{monomialideal}. The continuity of $\phi_{<'}$ means that $\phi_{<'}(F)$ will be contained in a facet of $\bigcup_i \Delta_{\prec_i,<'}$. On $\Delta_{\prec_l,<'}$ the inverse $\phi_{<'}^{-1}$ is given by \[\phi_{<'}^{-1}(x)_p=\sum_{i=0}^m x_{p_i}\] where $p_0=p$, for $i<m$ the element $p_{i+1}$ is $\prec_l$-minimal such that $p_i<'p_{i+1}$ and $p_m$ is maximal in $(P,<')$. Indeed, this map agrees with $\phi_{<'}^{-1}$ in the vertices of $\Delta_{\prec_l,<'}$, i.e.\ sends $\one_{\max_{<'}J}$ to $\one_J$ for $J\in\cJ(P,\prec_l)$. Note that here $p_{i+1}$ necessarily covers $p_i$ with respect to $<'$ so that $p_0,\dots,p_m$ is a saturated chain in $(P,<')$.

This means that the hyperplane containing $\phi_{<'}(F)$ (which is a bounding hyperplane for $\bigcup_i \Delta_{\prec_i,<'}$) is given by the equation
\begin{equation}\label{facet}
(x_q+\sum_{i=1}^m x_{q_i})-(x_r+\sum_{j=1}^{m'} x_{r_j})=0
\end{equation}
where the chains $q,q_1,\dots,q_m$ and $r,r_1,\dots,r_m$ are of the discussed form.
Since $q\notin J$ and $J\in\cJ'$, all $q_i\notin J$ as well. Furthermore, since $r\in J$ and $r_m$ is maximal in $(P,<')$, we see that precisely one $r_j$ lies in $\max_{<'} J$. As a result we see that when $x=\one_{\max_{<'} J}$ the left-hand side of~\eqref{facet} evaluates to $-1$. However, by construction for any point $x\in\Delta_{\prec_l,<}$ the left-hand side of~\eqref{facet} will be nonnegative. By convexity the latter must hold for any $x\in\bigcup_i \Delta_{\prec_i,<'}$, in particular, for $\one_{\max_{<'} J}$. We arrive at a contradiction.
\end{proof}

\begin{theorem}\label{partsrpp}
For $w\in\overline{C(\Psi(P,<,<'))}$ all parts of the subdivision $\Theta_{\cR(P,<,<')}(w)$ have the form $\cR(P,<'',<')$ for some $<''$ stronger than $<$.
\end{theorem}
\begin{proof}
Choose a part $Q$ of $\Theta_{\cR(P,<,<')}(w)$ and let $\mathcal K\subset\cJ$ denote the set of $J$ for which $\one_{\max_{<'}J}\in Q$. We are to prove that $\mathcal K$ is closed under $\cup$, $\cap$ and $*'$. Indeed, closedness under $\cup$ and $\cap$ would imply that $\mathcal K$ is a sublattice. Since $Q$ is a union of simplices $\Delta_{\prec,<'}$, the set $\mathcal K$ must contain at least one maximal chain in $\cJ$ and be a sublattice of maximal height. Proposition~\ref{sublattices} then shows that $\mathcal K=\cJ(P,<'')$ for some $<''$ stronger than $<$ and closedness under $*'$ mean that $Q$ is indeed the relative poset polytope $\cR(P,<'',<')$.

First we show that $\mathcal K$ is closed under $\cup$ and $*'$. Let $f$ be the convex piecewise linear function on $\cR(P,<,<')$ with $f(\one_{\max_{<'}J})=w_J$ for all $J\in\cJ$ the domains of linearity of which are the parts of $\Theta_{\cR(P,<,<')}(w)$. Let $g$ be the affine function that agrees with $f$ on $Q$, then for $x\notin Q$ we have $f(x)<g(x)$. Hence, if $J_1,J_2\in\mathcal K$ but at least one of $J_1\cup J_2$ and $J_1*'J_2$ is not in $\mathcal K$, then we must have 
\begin{multline*}
w_{J_1}+w_{J_2}=g(\one_{\max_{<'}J_1})+g(\one_{\max_{<'}J_2})=g(\one_{\max_{<'}(J_1\cup J_2)})+g(\one_{\max_{<'}(J_1*'J_2)})>\\f(\one_{\max_{<'}(J_1\cup J_2)})+f(\one_{\max_{<'}(J_1*'J_2)})=w_{J_1\cup J_2}+w_{J_1*'J_2}
\end{multline*}
which contradicts $w\in\overline{C(\Psi(P,<,<'))}$ by Proposition~\ref{secondarycone}.

Now, to show that $\mathcal K$ is closed under $\cap$ note that by Lemma~\ref{convexunion} the polytope $\phi_{<'}^{-1}(Q)$ is convex. Since $\phi_{<'}^{-1}(Q)$ is a union of the simplices $\cO(P,\prec)$, it is the intersection of semispaces of the form $\{x|x_p\le x_q\}$ with $p,q\in P$. It is easily seen, however, that every such intersection with at least one interior point is necessarily an order polytope: it has the form $\cO(P,<'')$ for some $<''$. Then $\mathcal K$ must coincide with $\cJ(P,<'')$ and, in particular, be closed under $\cap$. Of course, the order $<''$ is precisely the one appearing in the lemma.
\end{proof} 

We also have a geometric counterpart of the above theorem.
\begin{theorem}\label{maindegen}
For any $w\in\overline{C(I_{P,<,<'},I^m_{P,<})}$ all irreducible components of the zero set of $\initial_w I_{P,<,<'}$ in $\bP(\bC^\cJ)$ have the form $H(P,<'',<')$ for some $<''$ stronger than $<$. Here $H(P,<'',<')\subset\bP(\bC^{\cJ(P,<'')})$ is embedded into $\bP(\bC^\cJ)$ via the natural embedding $\cJ(P,<'')\subset\cJ$.
\end{theorem}
\begin{proof}
By Theorem~\ref{gfans}(d) the radical ideal $I^m_{P,<}$ is an initial ideal of $\initial_w I_{P,<,<'}$, hence the latter is also radical. Recall that $I_{P,<,<'}=I_{V(\cR(P,<,<'))}$ and thus by Theorem~\ref{zhu} we have \[\initial_w I_{P,<,<'}=\bigcap_{Q\in\Theta_{\cR(P,<,<')}(w)} I_{V(Q),V(\cR(P,<,<'))}.\] However, we have seen that every $Q\in\Theta_{\cR(P,<,<')}(w)$ has the form $\cR(P,<'',<')$ for some $<''$ stronger than $<$. The ideal $I_{V(Q),V(\cR(P,<,<'))}$ is generated by $I_{P,<'',<'}\subset\bC[\cJ(P,<'')]$ and the $X_J$ with $J\notin\cJ(P,<'')$. Consequently, the zero set of $I_{V(Q),V(\cR(P,<,<'))}$ is precisely $H(P,<'',<')\subset\bP(\bC^{\cJ(P,<'')})\subset\bP(\bC^\cJ)$.
\end{proof}

In other words, every Gr\"obner degeneration of the relative Hibi variety $H(P,<,<')$ which degenerates further into $M(P,<)$ is a semitoric variety the components of which are themselves relative Hibi varieties. In particular, in the special case when $<'$ is the trivial order, we have the following fact (which can, of course, be deduced without the use of relative poset polytopes).
\begin{cor}\label{ordercase}
For any $w\in\overline{C(I^h_{P,<},I^m_{P,<})}$ all irreducible components of the zero set of $\initial_w I_{P,<}$ in $\bP(\bC^\cJ)$ have the form $H(P,<'')$ for some $<''$ stronger than $<$.
\end{cor}

\begin{remark}\label{permutahedra}
Let the order $<$ be trivial so that $\cJ$ is the power set of $P$ and $\cO(P,<)$ is the unit cube. Functions $u:\cJ\to\bR$ satisfying \[u(J_1)+u(J_2)\ge u(J_1\cap J_2)+u(J_1\cup J_2)\] are known as \emph{submodular functions}. Evidently, the set of submodular functions can be viewed as a polyhedral cone in $\bR^\cJ$ (known as the~\emph{submodular cone}) which coincides with $-\overline{C(I^h_{P,<},I^m_{P,<})}$. Submodular functions are in correspondence (\cite[Proposition 12]{MPSSW}) with an important family of polytopes known as \emph{generalized permutahedra} (\cite{P}). This correspondence can be described in terms of the above construction as follows. For a submodular function $u$ denote $w=-u\in \overline{C(I^h_{P,<},I^m_{P,<})}$. Let $\Theta_{\cO(P,<)}(w)$ consist of parts $\cO(P,<_i)$. Let $f$ be the convex function on $\cO(P,<)$ with all $f(\mathbf 1_J)=w_J$ and linear on every $\cO(P,<_i)$. The restriction of $f$ to $\cO(P,<_i)$ has the form $(\alpha_i,x)+w_\varnothing$ and the convex hull of all $-\alpha_i$ is the generalized permutahedron corresponding to $u$. Furthermore, using the notions of \emph{extended submodular functions} and \emph{extended generalized permutahedra} (see~\cite[Subsection 12.4]{AA}) one could generalize this observation to subdivisions of arbitrary order polytopes. It would be interesting to explore the relationship between the semitoric degeneration given by $w$ and the corresponding generalized permutahedron.
\end{remark}

\subsection{Degenerations of Grassmannians}\label{grassmannians}

Two well known constructions provide toric degenerations of Grassmannians, one of them is given by an order polytope while the other is given by a chain polytope. This means that we can apply the above results to obtain two families of semitoric degenerations for every Grassmannian, here we provide the details.

For this subsection fix natural numbers $k\le n$ and let $\Gr(k,n)$ denote the Grassmannian of $k$-dimensional subspaces in $\bC^n$. The Pl\"ucker embedding realizes $\Gr(k,n)$ as a subvariety in $\bP(\wedge^k\bC^n)$. Since the homogeneous coordinates in the latter space are enumerated by $k$-subsets in $\{1,\dots,n\}$, we can view $\Gr(k,n)$ as the zero set of a certain ideal \[\tilde I_k\subset\bC[\{X_{a_1,\dots,a_k}\}_{1\le a_1<\dots<a_k\le n}]=S_k,\] this is the corresponding \textbf{Pl\"ucker ideal}.

Consider also the ideal $\tilde I_k^h\subset S_k$ generated by the binomials \begin{equation}\label{glbinomial}
X_{a_1,\dots,a_k}X_{b_1,\dots,b_k}-X_{\min(a_1,b_1),\dots,\min(a_k,b_k)}X_{\max(a_1,b_1),\dots,\max(a_k,b_k)}.
\end{equation}
The following is (a special case of) one of the first results concerning flat degenerations of flag varieties.
\begin{theorem}[\cite{GL}]\label{gl}
The ideal $\tilde I_k^h$ is an initial ideal of $\tilde I_k$.
\end{theorem}

Next, let $(P_k,<)$ be the poset consisting of elements $p_{i,j}$ with $1\le i\le k$ and $k+1\le j\le n$ where $p_{i_1,j_1}<p_{i_2,j_2}$ if and only if $i_1\le i_2$ and $j_1\le j_2$. Consider the map $\psi_\cO$ from $S_k$ to $\bC[\cJ(P_k,<)]$ given by 
\begin{equation}\label{psiO}
\psi_\cO(X_{a_1,\dots,a_k})=X_{\{p_{i,j}\mid j\le a_{k+1-i}+i-1\}}.
\end{equation}
\begin{example}
Below is the Hasse diagram of the poset $(P_3,<)$ for $n=7$. Elements of the order ideal $J$ for which $\psi_\cO(X_{2,4,7})=X_J$ are colored cyan. 
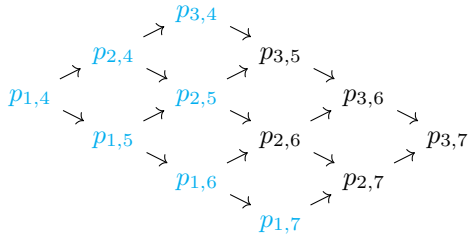
\[
\begin{tikzcd}[row sep=.5mm,column sep=2.5mm]
&&{\color{cyan} p_{3,4}}\arrow[dr]\\
&{\color{cyan}p_{2,4}}\arrow[dr]\arrow[ur]&&  p_{3,5}\arrow[dr]\\
{\color{cyan}p_{1,4}}\arrow[dr]\arrow[ur]&&\color{cyan}p_{2,5}\arrow[dr]\arrow[ur] &&   p_{3,6}\arrow[dr]\\
&{\color{cyan}p_{1,5}}\arrow[ur]\arrow[dr]&&p_{2,6}\arrow[dr]\arrow[ur] && p_{3,7}\\
&&{\color{cyan}p_{1,6}}\arrow[dr]\arrow[ur]&&p_{2,7}\arrow[ur]\\
&&&\color{cyan}p_{1,7}\arrow[ur]
\end{tikzcd}
\]
\end{example}

\begin{proposition}\label{psiOisom}
$\psi_\cO$ is an isomorphism and $\psi_\cO(\tilde I_k^h)=I^h_{P_k,<}$.
\end{proposition}
\begin{proof}
For an order ideal $J$ in $(P_k,<)$ and $i\in[1,k]$ let $c_i$ equal the largest $j$ such that $p_{i,j}\in J$. Than $c_1,\dots,c_k$ is a weakly decreasing sequence of integers in $[k+1,n]$ and the numbers $b_i=c_i-(i-1)$ form a strictly decreasing sequence with $\psi_\cO(X_{b_k,\dots,b_1})=X_J$. It is also evident that any such strictly decreasing sequence can be obtained from some ideal $J$, i.e.\ $\psi_\cO$ is a bijection between the variable sets.

The second claim follows from the straightforward fact that $\psi_\cO$ maps the binomial~\eqref{glbinomial} to the generator~\eqref{hibibinomial} of the Hibi ideal with $J_1=\{p_{i,j}|j\le a_{k+1-i}+i-1\}$ and $J_2=\{p_{i,j}|j\le b_{k+1-i}+i-1\}$.
\end{proof} 

The observation that the Hibi ideal of $(P_k,<)$ is the toric ideal of Gonciulea--Lakshmibai is due to~\cite{FL}. In particular, this means that the flat degeneration provided by Theorem~\ref{gl} is the toric variety of $\cO(P_k,<)$ (which can also be viewed as the Gelfand--Tsetlin polytope of a fundamental $\mathfrak{gl}_n$-weight, see~\cite{KM}). The proposition also implies that $\tilde I_k^m=\psi_\cO^{-1}(I^m_{P_k,<})$ is an initial ideal of $\tilde I_k^h$ and, subsequently, of $\tilde I_k$. From Corollary~\ref{ordercase} we now obtain
\begin{theorem}\label{grass1}
Let $\tilde I\subset S_k$ be an initial ideal of $\tilde I_k^h$ such that $\tilde I_k^m$ is an initial ideal of $\tilde I$. Then the zero set of $\tilde I$ in $\bP(\wedge^k\bC^n)$ is a flat degeneration of $\Gr(k,n)$ such that all of its irreducible components are Hibi varieties of the form $H(P_k,<'')$ with $<''$ stronger than $<$.
\end{theorem}

It was later realized that the chain polytope $\cC(P_k,<)$ (which can be viewed as the Feigin--Fourier--Littelmann--Vinberg polytope of a fundamental weight, see~\cite{FFL1}) also provides a toric degeneration of the Grassmannian. Indeed, let us define a different isomorphism from $S_k$ to $\bC[\cJ(P_k,<)]$. First, let $\alpha_1,\dots,\alpha_k$ be pairwise distinct integers in $[1,n]$ and let $\sigma$ be a permutation of $[1,k]$. We use the standard notation \[X_{\alpha_1,\dots,\alpha_k}=(-1)^\sigma X_{\alpha_{\sigma(1)},\dots,\alpha_{\sigma(k)}} \in S_k\] to define $X_{\alpha_1,\dots,\alpha_k}$ with not necessarily increasing subscripts. Now choose $J\in\cJ(P_k,<)$ and let $p_{i_1,j_1},\dots,p_{i_l,j_l}$ be the maximal elements in $J$. Define $\alpha_1,\dots,\alpha_k$ by setting $\alpha_{i_r}=j_r$ for $r\in[1,l]$ and setting $\alpha_i=i$ for all other $i$. We set \[\psi_\cC(X_{\alpha_1,\dots,\alpha_k})=X_J.\] It is obvious that distinct $J$ give distinct sets $\{\alpha_1,\dots,\alpha_k\}$ and we have already seen that $|\cJ(P_k,<)|={n\choose k}$, hence $\psi_\cC$ is indeed an isomorphism from $S_k$ to $\bC[\cJ(P_k,<)]$.
\begin{example}
Below elements of the order ideal $J$ for which $\psi_\cC(X_{7,6,3})=X_J$ are colored cyan. 
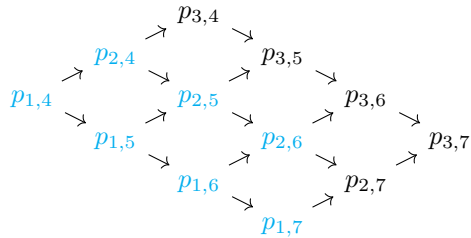
\[
\begin{tikzcd}[row sep=.5mm,column sep=2.5mm]
&&p_{3,4}\arrow[dr]\\
&{\color{cyan}p_{2,4}}\arrow[dr]\arrow[ur]&&  p_{3,5}\arrow[dr]\\
{\color{cyan}p_{1,4}}\arrow[dr]\arrow[ur]&&\color{cyan}p_{2,5}\arrow[dr]\arrow[ur] &&   p_{3,6}\arrow[dr]\\
&{\color{cyan}p_{1,5}}\arrow[ur]\arrow[dr]&&\color{cyan}p_{2,6}\arrow[dr]\arrow[ur] && p_{3,7}\\
&&{\color{cyan}p_{1,6}}\arrow[dr]\arrow[ur]&&p_{2,7}\arrow[ur]\\
&&&\color{cyan}p_{1,7}\arrow[ur]
\end{tikzcd}
\]
\end{example}

Let $\tilde I_k^{hl}\subset S_k$ denote the ideal $\psi_\cC^{-1}(I^{hl}_{P_k,<})$. The following fact is originally due to~\cite{FFL2}, for a form closer to the below, see~\cite[Theorem 5.1]{fffm}.
\begin{theorem}
The ideal $\tilde I_k^{hl}$ is an initial ideal of $\tilde I_k$.
\end{theorem}

Similarly to the above, we can now apply Theorem~\ref{maindegen} and obtain
\begin{theorem}\label{grass2}
Let $\tilde I\subset S_k$ be an initial ideal of $\tilde I_k^{hl}$ such that $\tilde I_k^m$ is an initial ideal of $\tilde I$. Then the zero set of $\tilde I$ in $\bP(\wedge^k\bC^n)$ is a flat degeneration of $\Gr(k,n)$ such that all of its irreducible components are relative Hibi varieties of the form $H(P_k,<'',<)$.
\end{theorem}

This theorem can be said to make essential use of the notion of relative poset polytopes unlike Theorem~\ref{grass1}. That is since in Theorem~\ref{grass2} we may obtain irreducible components which are neither Hibi varieties nor Hibi--Li varieties but more general relative Hibi varieties (see also Example~\ref{notmcop}).

\section{Marked relative poset polytopes}\label{marked}

\subsection{Definition and main properties}

In this section we consider a finite poset $(P,<)$ with a subset $P^*\subset P$ of \textit{marked elements} which contains all minimal and maximal elements of $(P,<)$. We also choose $\la\in\bZ^{P^*}$ such that for elements $p<q$ of $P^*$ we have $\la_p\ge \la_q$. In general, we will refer to elements of $\bZ^{P^*}$ as \textit{markings} and to those satisfying the above inequalities as \textit{dominant markings}. We call markings with only nonnegative coordinates \textit{nonnegative}. 

We also choose an order $<'$ on $P$ with the following three properties:
\begin{enumerate}[label=(\roman*)]\label{properties}
\item $<'$ is weaker than $<$,
\item $\cJ=\cJ(P,<)$ is closed under $*_{P,<'}$,
\item $p\not<' q$ for any $p\in P^*$ and $q\in P$ (in particular, all elements of $P^*$ are pairwise incomparable).
\end{enumerate}

\begin{remark}
We point out that the newly added property (iii) is easily satisfied: one may choose any order with properties (i) and (ii) (so of the type considered in the previous section) and then simply remove all relations contradicting (iii).
\end{remark}

The above data defines a marked relative poset polytope, however, we first consider the case when all coordinates of $\la$ are 0 or 1, i.e.\ $\la$ has the form $\mathbf 1_K$ for some $K\in\cJ^*=\cJ(P^*,<)$. We call markings of this form \textit{fundamental} and denote them $\omega_K=\mathbf 1_K$.
\begin{definition}\label{funddef}
For $K\in\cJ^*$ the \textit{fundamental marked relative poset polytope}\linebreak (\textit{fundamental MRPP}) $\cR_{\omega_K}(P,<,<')\subset\bR^P$ is the convex hull of points $\mathbf 1_{\max_{<'}J}$ for $J\in\cJ$ such that $J\cap P^*=K$.
\end{definition}
The first thing to note is that all points $x\in\cR_{\omega_K}(P,<,<')$ satisfy $x_p=(\om_K)_p$ for all $p\in P^*$. Also, $\cR_{\omega_K}(P,<,<')$ is always nonempty because we can always choose $J$ to be the order ideal in $(P,<)$ generated by $K$. In the extremal cases we obtain polytopes consisting of a single point: $\cR_{\omega_\varnothing}(P,<,<')=\{0\}$ and $\cR_{\omega_{P^*}}(P,<,<')=\{\mathbf 1_{\max_{<'}P}\}$, since $P^*$ contains all the minimal and maximal elements. 
\begin{example}\label{fundexample}
Every relative poset polytope $\cR(Q,\ll,\ll')$ can be identified with a fundamental MRPP. Let $P=Q\cup\{p_0,p_1\}$ and let $<$ coincide with $\ll$ on $Q$ while $p_0$ is the unique minimal element and $p_1$ is the unique maximal element in $(P,<)$. Furthermore, let $<'$ coincide with $\ll'$ on $Q$ while $p_0$ and $p_1$ are $<'$-incomparable with all other elements. Set $P^*=\{p_0,p_1\}$ and $K=\{p_0\}$. Then one sees that the projection from $\bR^P$ to $\bR^Q$ identifies $\cR_{\om_K}(P,<,<')$ with $\cR(Q,\ll,\ll')$.

In fact, it is also true that any fundamental MRPP can be identified with a relative poset polytope. For $K\in\cJ^*$ let $P_0\subset P$ consist of $p$ for which there is no $q\in K$ with $p\le q$ and no $q\in P^*\backslash K$ with $q\le p$. Then one may check that the projection from $\bR^P$ to $\bR^{P_0}$ identifies $\cR_{\omega_K}(P,<,<')$ with $\cR(P_0,<,<')$. This shows that the notions of fundamental MRPPs and of relative poset polytopes are essentially equivalent.
\end{example}

We have the following characterization equivalent to the definition.
\begin{proposition}\label{fundassection}
$\cR_{\omega_K}(P,<,<')$ is the intersection of $\cR(P,<,<')$ with the affine subspace $U_K$ of points $x$ with $x_p=(\omega_K)_p$ for all $p\in P^*$. This intersection is a face of $\cR(P,<,<')$. 
\end{proposition}
\begin{proof}
Consider the hyperplane \[H=\Big\{x\in\bR^P|\sum_{p\in K}x_p-\sum_{p\in P^*\backslash K}x_p=|K|\Big\}.\] Evidently, $H$ is a supporting hyperplane of $\cR(P,<,<')$ and \[H\cap\cR(P,<,<')=U_K\cap \cR(P,<,<').\] This shows that $U_K\cap \cR(P,<,<')$ is a face of the latter. The only vertices of $\cR(P,<,<')$ contained in $H$ are $\mathbf 1_{\max_{<'}J}$ for $J\in\cJ$ such that $J\cap P^*=K$. This shows that \[U_K\cap \cR(P,<,<')=\cR_{\omega_K}(P,<,<').\qedhere\]
\end{proof}

The next statement is needed to generalize the definition to general $\la$.
\begin{proposition}\label{markingdecomp}
Let $\mu$ be a nonnegative dominant marking. Then there exists a unique decomposition of the form $\mu=\sum_{i=1}^m \alpha_i\om_{K_i}$ where $\alpha_i\in\bZ_{\ge0}$ and $K_i\in\cJ^*$ satisfy \[\varnothing\neq K_1\subsetneq\dots\subsetneq K_m.\]
\end{proposition}
\begin{proof}
Set $M=\max_p\mu_p$. Then $\mu$ is an integer point in the dilated order polytope $M\cO(P^*,<)$ and the desired decomposition and its uniqueness are obtained from Proposition~\ref{inctuple} by collecting like terms.
\end{proof}
We will denote the coefficient of $\om_K$ in the above decomposition by $\alpha_K(\mu)$. Recall that $\la$ was chosen to be dominant.
\begin{definition}\label{markeddef}
If $\la$ is nonnegative, then the \textit{marked relative poset polytope} (or \textit{MRPP}) $\cR_\la(P,<,<')$ is the Minkowski sum 
\[
\sum_{K\in\cJ^*} \alpha_K(\la)\cR_{\om_K}(P,<,<').
\]
If $\la$ has negative coordinates, choose an integer $\beta$ such that all coordinates of $\la+\beta\om_{P^*}$ are nonnegative and define \[\cR_\la(P,<,<')=\cR_{\la+\beta\om_{P^*}}(P,<,<')-\beta\mathbf 1_{\max_{<'}P}.\]
\end{definition}

To see that in the second case the polytope does not depend on the choice of $\beta$, consider two nonnegative markings $\mu_1$ and $\mu_2$ with $\mu_1-\mu_2=\gamma\om_{P^*}$ for some $\gamma\in\bZ$. We have $\alpha_{P^*}(\mu_1)-\alpha_{P^*}(\mu_2)=\gamma$ while $\alpha_K(\mu_1)=\alpha_K(\mu_2)$ for all other $K$. The independence now follows from $\cR_{\om_{P^*}}(P,<,<')=\{\mathbf 1_{\max_{<'}P}\}$.

Again, one notices that $x_p=\la_p$ for all $x\in\cR(P,<,<')$ and $p\in P^*$. We can generalize Proposition~\ref{fundassection} by describing general MRPPs as affine sections. We give such a description in the case of nonnegative $\la$ (which suffices for our needs), if $\la$ has negative coordinates, one may easily obtain a similar statement translating by the appropriate multiple of $\mathbf 1_{\max_{<'}P}$.
\begin{proposition}\label{assection}
Suppose $\la$ is nonnegative and let $S=\sum_K\alpha_K(\la)$. Then $\cR_\la(P,<,<')$ is the intersection of $S\cR(P,<,<')$ with the subspace $U_\la$ of points $x$ with $x_p=\la_p$ for all $p\in P^*$.
\end{proposition}
\begin{proof}
The inclusion \[\cR_\la(P,<,<')\subset S\cR(P,<,<')\cap U_K\] is immediate from the definition and Proposition~\ref{fundassection}. Now consider a rational point $x\in S\cR(P,<,<')\cap U_K$. For a large enough integer $N$ the point $Nx\in NS\cR(P,<,<')$ is integer and may uniquely be decomposed as \[x=\sum_{i=1}^{NS}\mathbf 1_{\max_{<'} J_i}\] with $J_i\in\cJ$ forming a weakly increasing tuple by Proposition~\ref{inctuple}. This implies \[N\la=\sum_{i=1}^{NS}\mathbf 1_{J_i\cap P^*}.\] However, we have $N\la=\sum_{K\in\cJ^*}N\alpha_K(\la)\mathbf 1_K$ and, by uniqueness in Proposition~\ref{markingdecomp}, these two decompositions must coincide. Hence, for every $K\in\cJ^*$ precisely $N\alpha_K(\la)$ order ideals $J_i$ satisfy $J_i\cap P^*=K$ which implies $x\in\cR_\la(P,<,<')$.
\end{proof}

\begin{lemma}
If $\la$ is nonnegative, every integer point $x\in\cR_\la(P,<,<')$ can be decomposed as $x=\sum_{K\in\cJ^*} x_K$ with $x_K$ an integer point in $\alpha_K(\la)\cR_{\om_K}(P,<,<')$.
\end{lemma}
\begin{proof} 
Again denote $S=\sum_K\alpha_K(\la)$. Repeating the argument from the previous proof for the point $x$ and $N=1$ we obtain a tuple $J_1\subset\dots\subset J_S$ in $\cJ$ with \[x=\mathbf 1_{\max_{<'}J_1}+\dots+\mathbf 1_{\max_{<'}J_S}\] where for every $K\in\cJ^*$ exactly $\alpha_K(\la)$ order ideals $J_i$ satisfy $J_i\cap P^*=K$. The lemma follows, since $\sum_{J_i\cap P^*=K}1_{\max_{<'}J_i}$ is an integer point in $\alpha_K(\la)\cR_{\om_K}(P,<,<')$.
\end{proof}

We have the following consequences. 

\begin{cor}\label{fundinctuple}
If $\la$ is nonnegative and $S=\sum_K \alpha_K(\la)$, every integer point $x\in\cR_\la(P,<,<')$ can be uniquely expressed in the form \[x=\mathbf 1_{\max_{<'}J_1}+\dots+\mathbf 1_{\max_{<'}J_S}\] where $J_1\subset\dots\subset J_S$ is a weakly increasing tuple in $\cJ$ and exactly $\alpha_K(\la)$ order ideals $J_i$ satisfy $J_i\cap P^*=K$.
\end{cor}

\begin{cor}
The polytope $\cR_\la(P,<,<')$ is normal.
\end{cor}
\begin{proof}
It suffices to consider the case of a nonnegative $\la$ because normality is preserved by an integer translation. However, for $m\in\bZ_{\ge 0}$ we have $m\cR_\la(P,<,<')=\cR_{m\la}(P,<,<')$ and $\alpha_K(m\la)=m\alpha_K(\la)$ for all $K$. Hence, given an integer point $x\in m\cR_\la(P,<,<')$ we may apply Corollary~\ref{fundinctuple} to $x$ and then regroup the summands to obtain a sum of $m$ integer points in $\cR_\la(P,<,<')$.
\end{proof}

\begin{cor}
Let $<''$ be another order on $P$ satisfying conditions (i)-(iii). Then the polytopes $\cR_\la(P,<,<')$ and $\cR_\la(P,<,<'')$ are Ehrhart-equivalent.
\end{cor}
\begin{proof}
By Corollary~\ref{fundinctuple}, for any $m\in\bZ_{\ge 0}$ the number of integer points in $m\cR_\la(P,<,<')$ is equal to the number of weakly increasing tuples in $\cJ$ containing exactly $\alpha_K(m\la)$ order ideals $J$ satisfying $J\cap P^*=K$ for every $K\in\cJ^*$. However, the same holds for $m\cR_\la(P,<,<'')$ and the claim follows.
\end{proof}

\begin{remark}\label{chainorder}
Consider the two extreme cases when $<'$ is trivial and when $<'$ is obtained from $<$ by only removing relations that contradict (iii). One may check that $\cR_\la(P,<,<')$ then coincides with, respectively, the \textit{marked chain polytope} and the \textit{marked order polytope} defined in~\cite{ABS}. Hence, any $\cR_\la(P,<,<'')$ is Ehrhart-equivalent to both of these polytopes.
\end{remark}

\subsection{Toric varieties of MRPPs}\label{MRPPtoric}

Denote by $H_\la(P,<,<')$ the projective toric variety of $\cR_\la(P,<,<')$, we call it the \textbf{marked relative Hibi variety}. In this section we obtain a multiprojective realization of $H_\la(P,<,<')$ and interpret the defining ideal as a relative Hibi ideal. We also deduce that every MRPP can be identified with another MRPP of a very special form. 

\begin{definition}
We attach a chain $D(\la)$ in the lattice $\cJ^*$ to the marking $\la$. If $\la$ is nonnegative than $D(\la)$ consists of those $K$ for which $\alpha_K(\la)>0$ together with $\varnothing$ and $P^*$. For general $\la$ we set $D(\la)=D(\la+\beta\om_{P^*})$ where the latter marking is nonnegative.
\end{definition}

\begin{proposition}\label{samefan}
Consider a dominant marking $\mu$ such that $D(\mu)=D(\la)$, then the varieties $H_\la(P,<,<')$ and $H_\mu(P,<,<')$ are isomorphic.
\end{proposition}
\begin{proof}
This is immediate from Definition~\ref{markeddef} and the fact that the normal fan of a Minkowski sum is the common refinement of the summands' normal fans.
\end{proof}

For $K\in D(\la)$ let $V_K$ denote the set of integer points in $R_{\om_K}(P,<,<')$. Denote \[\bP_\la=\bigtimes_{K\in D(\la)}\bP(\bC^{V_K}).\] Let $\cJ_\la\subset\cJ$ denote the sublattice of ideals $J$ for which $J\cap P^*\in D(\la)$. Note that $V_K$ is in bijection with the set of $J\in\cJ_\la$ with $J\cap P^*=K$ and we, therefore, have a bijection between $\bigcup_{K\in D(\la)} V_K$ and $\cJ_\la$. This lets us view $\bC[\cJ_\la]\subset\bC[\cJ]$ as the multihomogeneous coordinate ring of $\bP_\la$. Recall the ideal $I_{P,<,<'}\subset\bC[\cJ]$ and denote $I_\la=I_{P,<,<'}\cap\bC[\cJ_\la]$.
\begin{theorem}\label{mrppmultiproj}
The ideal $I_\la$ in $\bC[\cJ_\la]$ is multihomogeneous and its zero set in $\bP_\la$ is isomorphic to $H_\la(P,<,<')$. Moreover, $I_\la$ is the vanishing ideal of its zero set.
\end{theorem}
\begin{proof}
By Proposition~\ref{samefan} we may assume that $\la$ is nonnegative and $\alpha_K(\la)=1$ for all nonempty $K\in D(\la)$. Consider the polynomial ring $R=\bC[\{z_p\}_{p\in P},\{t_K\}_{K\in D(\la)}]$. By Lemma~\ref{multiproj}, $H_\la(P,<,<')$ is the zero set of the kernel of $\varphi_\la:\bC[\cJ_\la]\to R$ such that for $J\in\cJ_\la$ with $J\cap P^*=K$ one has \[\varphi_\la(X_J)=t_K\prod_{p\in\max_{<'}J}z_p.\]

To prove the proposition we show that the kernel of $\varphi_\la$ is $I_\la$. Indeed, $I_\la$ is the kernel of the restriction to $\bC[\cJ_\la]$ of the map $\varphi_{P,<,<'}$ defined in~\eqref{relkernel}. Obviously, the latter kernel contains $\ker\varphi_\la$. To prove the reverse inclusion denote the elements of $D(\la)$ by 
\begin{equation}\label{Delements}
\varnothing=K_0\subset\dots\subset K_{|D(\la)|-1}=P^*
\end{equation}
and choose $p_i\in K_i\backslash K_{i-1}$ for every $i\in[1,|D(\la)|-1]$. Set $\tau(t)=t_\varnothing$, $\tau(z_{p_i})=t_{K_i}t_{K_{i-1}}^{-1}z_{p_i}$ and $\tau(z_p)=z_p$ for all other $p$, then $\tau\varphi_{P,<,<'}=\varphi_\la$ on $\bC[\cJ_\la]$. 
\end{proof}

We proceed to show that the defining ideal $I_\la$ can itself be viewed as a relative Hibi ideal. Consider an equivalence relation $\sim$ on $P$ with $p_1\sim p_2$ when any $J\in\cJ_\la$ contains either both or neither of $p_1$ and $p_2$. Denote $Q=P/\sim$ and let $\pi:P\to Q$ be the natural projection. It is obvious that every $J\in\cJ_\la$ is a union of equivalence classes of $\sim$. Define an order $\ll$ on $Q$ setting $q_1\ll q_2$ if and only if every $J\in\cJ_\la$ with $\pi^{-1}(q_2)\subset J$ also contains the equivalence class $\pi^{-1}(q_1)$. It is a general and straightforward property of sublattices that $\pi$ defines a lattice isomorphism from $\cJ_\la$ to $\cJ(Q,\ll)$, in terms of Remark~\ref{categorical} the resulting embedding $\cJ(Q,\ll)\hookrightarrow\cJ$ is dual to the epimorphism $\pi:(P,<)\to(Q,\ll)$. Set $Q^*=\pi(P^*)$ and define another order $\ll'$ on $Q$ by $q_1\ll' q_2$ if and only if $q_1\notin Q^*$ and $p_1<' p_2$ for some $p_1\in\pi^{-1}(q)$ and $p_2\in\pi^{-1}(q_2)$. The antisymmetricity of $\ll'$ is ensured by all equivalence classes being convex with respect to $<$ and hence $<'$. Evidently, $\ll'$ is weaker than $\ll$.
\begin{proposition}\label{pistar}
For any $J_1,J_2\in\cJ_\la$ we have \[\pi(J_1*' J_2)=\pi(J_1)*_{Q,\ll'}\pi(J_2).\]
\end{proposition}
\begin{proof}
It should first be noted that $\cJ_\la$ is closed under $*'$ because $(J_1*'J_2)\cap P^*=J_1\cap J_2\cap P^*$. Now we apply formula~\eqref{star}. Consider \[q\in\max\nolimits_{\ll'}(\pi(J_1)*_{Q,\ll'}\pi(J_2)),\] we see that $q\in\pi(J_1\cap J_2)$ and that $q$ is maximal in one of $\pi(J_1)$ and $\pi(J_2)$, say $\pi(J_1)$. By the definition of $\ll'$ this means that either $q\notin Q^*$ and any $p\in \max_{<'}\pi^{-1}(q)$ is $<'$-maximal in $J_1$ or $\pi^{-1}(q)$ contains some $p\in P^*$ which must be $<'$-maximal in $J_1$. Since $p\in J_1\cap J_2$ we deduce that $p$ and hence all of $\pi^{-1}(q)$ lies in $J_1*' J_2$. This proves that the left-hand side of the desired equality contains the right-hand side.

To prove the reverse inclusion we are to show that every $q\in\max_{\ll'}\pi(J_1*'J_2)$ lies in $\pi(J_1)*_{Q,\ll'}\pi(J_2)$. If $q\in Q^*$, then by the definition of $\ll'$ it is $\ll'$-maximal in each of $\pi(J_1)$ and $\pi(J_2)$ and hence lies in $\pi(J_1)*_{Q,\ll'}\pi(J_2)$. Suppose $q\notin Q^*$. Let us show that $\pi^{-1}(q)$ consists of a single element. Indeed, suppose we have distinct $p_1,p_2\in\pi^{-1}(q)$ where $p_1\in\max_<{\pi^{-1}(q)}$. Consider the minimal $J_0\in\cJ_\la$ containing $\pi^{-1}(q)$, we claim that $p_1\in\max_< J_0$. Indeed, $J_0$ consists of those $p$ which lie in every $J\in\cJ_\la$ containing $p_1$ and if $p_0>p_1$ for some $p_0\in J_0$, then $p_0\sim p_1$, i.e.\ $p_0\in\pi^{-1}(q)$ which contradicts our choice of $p_1$. We deduce that $J_0\backslash\{p_1\}\in\cJ_\la$ which contradicts $p_1\sim p_2$. Now let $p$ be the only element of $\pi^{-1}(q)$, by our choice of $q$ we have $p\in\max_{<'}(J_1*'J_2)$, hence $p$ is $<'$-maximal in at least one of $J_1$ and $J_2$. Assume $p\in\max_{<'}J_1$, then $q\in\max_{\ll'}\pi(J_1)$ and, since $q\in\pi(J_1)\cap\pi(J_2)$, we obtain $q\in\pi(J_1)*_{Q,\ll'}\pi(J_2)$.
\end{proof}

We will denote by $\xi$ the isomorphism $X_J\mapsto X_{\pi(J)}$ from $\bC[\cJ_\la]$ to $\bC[\cJ(Q,\ll)]$.
\begin{lemma}\label{ashibi}
We have $\xi(I_\la)=I_{Q,\ll,\ll'}$.
\end{lemma}
\begin{proof}
In view of Theorem~\ref{quadratic} and Proposition~\ref{pistar} it suffices to show that $I_\la$ is generated by the expressions \begin{equation}\label{binomial2}
X_{J_1}X_{J_2}-X_{J_1\cup J_2}X_{J_1*'J_2} 
\end{equation}
for all $J_1,J_2\in\cJ_\la$. All of these expressions lie in $I_\la$ since they lie in $I_{P,<,<'}$. 

Consider a monomial $M=X_{J_1}\dots X_{J_m}\in\bC[\cJ_\la]$. If there are two incomparable $J_i$, say $J_1$ and $J_2$, then we may add a multiple of the binomial~\eqref{binomial2} to $M$ to obtain the monomial $X_{J_1\cup J_2}X_{J_1*'J_2}X_{J_3}\dots X_{J_m}$. By iterating this procedure we will arrive at a monomial $X_{J'_1}\dots X_{J'_m}$ with $J'_1\subset\dots\subset J'_m$. Hence $I_\la$ is generated by some set of linear combinations of such monomials $X_{J'_1}\dots X_{J'_m}$ and the binomials~\eqref{binomial2}. However, by Corollary~\ref{inctuple} for distinct $M'=X_{J'_1}\dots X_{J'_m}$ and $M''=X_{J''_1}\dots X_{J''_m}$ with $J'_1\subset\dots\subset J'_m$ and $J''_1\subset\dots\subset J''_m$ we have $\sum\one_{\max_{<'}J'_i}\neq\sum\one_{\max_{<'}J''_i}$ and, therefore, $\varphi_{P,<,<'}(M')\neq\varphi_{P,<,<'}(M'')$. Thus, a nontrivial linear combination of such monomials cannot lie in $I_\la$.
\end{proof}

As above let $D(\la)$ consist of $K_0\subset\dots\subset K_{|D(\la)|-1}$, for $i\in[1,|D(\la)|-1]$ denote $\Delta_i=K_i\backslash K_{i-1}$. Evidently, every $\Delta_i$ is contained in a single equivalence class of $\sim$ and elements contained in two distinct $\Delta_i$ cannot be equivalent. Also note that $\la_{p_1}=\la_{p_2}$ whenever $p_1,p_2\in\Delta_i$. We see that $Q^*$ consists of $|D(\la)|-1$ elements $\tilde q_i=\pi(\Delta_i)$ and we may define a dominant marking $\mu\in\bR^{Q^*}$ (for the poset $(Q,\ll)$)  by setting $\mu_{\tilde q_i}=\la_p$ where $p\in\Delta_i$. Proposition~\ref{pistar} implies that $\cJ(Q,\ll)$ is closed under $*_{Q,\ll'}$ and the definition of $\ll'$ shows that all elements of $Q^*$ are $\ll'$-maximal in $Q$, this lets us consider the MRPP $\cR_\mu(Q,\ll,\ll')$. We will now show that $\cR_\mu(Q,\ll,\ll')$ is unimodularly equivalent to $\cR_\la(P,<,<')$.

First let us note that the poset $(Q^*,\ll)$ is linearly ordered ($\tilde q_i\ll\tilde q_j$ whenever $i<j$). Every element of $\cJ(Q^*,\ll)$ has the form $L_i=\{\tilde q_1,\dots,\tilde q_i\}$ and we have $\alpha_{L_i}(\mu)=\alpha_{K_i}(\la)$. In particular, the chain $D(\mu)$ is all of $\cJ(Q^*,\ll)$. As above we can view $\bC[\cJ(Q,\ll)]$ as the multihomogeneous coordinate ring of \[\bP_\mu=\bigtimes_{i\in[1,|D(\la)|-1]}\bP(\bC^{V_{L_i}})\] where $V_{L_i}$ is the vertex set of $\cR_{\om_{L_i}}(Q,\ll,\ll')$ which is in bijection with the set of $M\in\cJ(Q,\ll)$ with $M\cap Q^*=L_i$. However, $\pi^{-1}$ provides a bijection between such $M$ and the $J\in\cJ$ with $J\cap P^*=K_i$ and, subsequently, with $V_{K_i}$. We conclude that there is a natural isomorphism between $\bP_\la$ and $\bP_\mu$ which identifies the zero set of $I$ with the zero set of $\zeta(I)$ for any multihomogeneous $I\subset\bC[\cJ_\la]$.

Now let us denote by $\Gamma_i$ the equivalence class containing $\Delta_i$. In the proof of Proposition~\ref{pistar} we have seen that all equivalence classes other than the $\Gamma_i$ consist of a single element. For every $i\in[1,|D(\la)|-1]$ choose any $\tilde p_i\in\Delta_i$ and define a projection $\theta:\bR^P\to\bR^Q$ by $\theta(x)_{\tilde q_i}=x_{\tilde p_i}$ for all $i$ and $\theta(x)_q=x_{\pi^{-1}(q)}$ for any $q\notin Q^*$ (and hence $\pi^{-1}(q)$ a unique element).

\begin{theorem}\label{standard}
$\theta$ is a unimodular equivalence from $\cR_\la(P,<,<')$ to $\cR_\mu(Q,\ll,\ll')$.
\end{theorem}
\begin{proof}
By Theorem~\ref{mrppmultiproj} the zero set of $I_{Q,\ll,\ll'}$ in $\bP_\mu$ is the toric variety of $\cR_\mu(Q,\ll,\ll')$. Hence the toric varieties of $\cR_\la(P,<,<')$ and $\cR_\mu(Q,\ll,\ll')$ are isomorphic and it suffices to show that $\theta$ provides a bijection between their integer point sets. In view of Corollary~\ref{fundinctuple} this reduces to showing that \[\theta(\one_{\max_{<'}J})=\one_{\max_{\ll'}\pi(J)}\] for any $J\in\cJ_\la$. Let $J\cap P^*=K_i$, then both $\theta(\one_{\max_{<'}J})_{\tilde q_j}$ and $(\one_{\max_{\ll'}\pi(J)})_{\tilde q_j}$ are 1 when $j\le i$ and 0 otherwise. If $q\notin Q^*$, then $q\in\max_{\ll'}\pi(J)$ if and only if $\pi^{-1}(q)\in\max_{<'}J$, which shows that the coordinates corresponding to $q$ also coincide.
\end{proof}

\begin{definition}
The MRPP $\cR_\la(P,<,<')$ is \textit{standard}, if the poset $(P^*,<)$ is linearly ordered and $D(\la)=\cJ$ (i.e.\ it is the unique maximal chain).
\end{definition}

Since $\cR_\mu(Q,\ll,\ll')$ is standard, we obtain the following fact.

\begin{cor}
Every MRPP is unimodularly equivalent to a standard MRPP.
\end{cor}


\subsection{Semitoric degenerations}

Similarly to Section~\ref{subdivisions} we construct a family of flat semitoric degenerations of $H_\la(P,<,<')$ with each irreducible component itself a marked relative Hibi variety. 

Consider the ideal $I_\la^m=I_{P,<}^m\cap\bC[\cJ_\la]$. It is generated by products $X_{J_1}X_{J_2}$ with $J_1,J_2\in\cJ_\la$ incomparable and we have $\zeta(I_\la^m)=I_{Q,\ll}^m$. In particular, Lemma~\ref{ashibi} shows that $I_\la^m$ is an initial ideal of $I_\la$. This lets us consider the cone $C(I_\la,I_\la^m)$, in view of Proposition~\ref{groebnercone} its closure consists of $w\in\cR^{\cJ_\la}$ satisfying all inequalities \[w_{J_1}+w_{J_2}\le w_{J_1\cup J_2}+w_{J_1*'J_2}\] with $J_1,J_2\in\cJ_\la$.

\begin{theorem}\label{mrppmain}
Consider $w\in\overline{C(I_\la,I_\la^m)}$ and the initial ideal $\initial_w I_\la$. There exist orders $\ll_1,\dots,\ll_m$ on $Q$ stronger than $\ll$ such that the zero set of $\initial_w I_\la$ in $\bP_\la$ is semitoric with irreducible components isomorphic to the $H_\mu(Q,\ll_i,\ll')$ and the polytopes $\cR_\mu(Q,\ll_i,\ll')$ form a polyhedral subdivision of $\cR_\mu(Q,\ll,\ll')$.
\end{theorem}
\begin{proof}
By Theorem~\ref{maindegen} we have partial orders $\ll_i$ on $Q$ such that 
\[
\initial_w I_{Q,\ll,\ll'}=\bigcap_{i=1}^m I_{Q,\ll_i,\ll'}.
\]
The irreducible components of the zero set of $\initial_w I_\la$ in $\bP_\la$ are the zero sets of the ideals $\zeta^{-1}(I_{Q,\ll_i,\ll'})$ which are isomorphic to the zero sets of the $I_{Q,\ll_i,\ll'}$ in $\bP_\mu$. By Theorem~\ref{mrppmultiproj} the latter zero sets are isomorphic to the $H_\mu(Q,\ll_i,\ll')$.

By Theorem~\ref{partsrpp} the relative poset polytopes $\cR(P,\ll_i,\ll')$ form a subdivision of $\cR(P,\ll,\ll')$. By Proposition~\ref{assection}, intersecting this subdivision with the subspace $U_\mu=\{x\in\bR^Q|\forall q\in Q^*: x_q=\mu_q\}$ we obtain a subdivision of $\cR_\mu(P,\ll,\ll')$ with parts $\cR_\mu(P,\ll_i,\ll')$.
\end{proof}


Since the set of radical ideals saturated with respect to the ``irrelevant'' ideals is closed under intersection, one sees that $\initial_w I_\la$ is the vanishing ideal of its zero set. This means that the above theorem describes a family of flat semitoric degenerations of the relative Hibi variety $H_\la(P,<,<')$.

We may consider the preimages $\theta^{-1}(\cR_\mu(Q,\ll_i,\ll'))$ under the unimodular equivalence in Theorem~\ref{standard}. These preimages form a polyhedral subdivision of the original polytope $\cR_\la(P,<,<')$, they are unimodularly equivalent to the $\cR_\mu(Q,\ll_i,\ll')$ and their toric varieties are the irreducible components of the zero set of $\initial_w I_\la$. However, the preimages do not seem to necessarily be MRPPs themself, i.e.\ there does not seem to be a way of lifting the orders $\ll_i$ to $P$ that would provide \[\cR_\la(P,<_i,<')=\theta^{-1}(\cR_\mu(P,\ll_i,\ll')).\] This leads us to consider the polytopes $\cR_\mu(Q,\ll_i,\ll')$ in the theorem. Nonetheless, when $\cR_\la(P,<,<')$ is standard, one sees that the relation $\sim$ is trivial (in view of $\cJ_\la=\cJ$) and we have $<=\ll$ and $<'=\ll'$ (since $Q=P$). In this case we also have $I_\la=I_{P,<,<'}$ and $I_\la^m=I_{P,<}^m$. Theorem~\ref{mrppmain} then has an especially nice form.

\begin{cor}\label{standardmain}
Let $\cR_\la(P,<,<')$ be standard and consider $w\in\overline{C(I_\la,I_\la^m)}$. There exist orders $<_1,\dots,<_m$ on $P$ stronger than $<$ such that the zero set of $\initial_w I_\la$ in $\bP_\la$ is semitoric with irreducible components isomorphic to the $H_\la(P,<_i,<')$ and the polytopes $\cR_\la(P,<_i,<')$ form a polyhedral subdivision of $\cR_\la(P,<,<')$.
\end{cor}

Here we emphasize that by Theorem~\ref{standard} any MRPP occurring in an application of our construction can be identified with a standard MRPP. This lets one apply the simpler statement found in Corollary~\ref{standardmain} when considering semitoric degenerations. The next subsection serves as an example of such an approach, there we interpret all Gelfand--Tsetlin and FFLV polytopes as standard MRPPs.

\subsection{Applications to flag varieties}\label{flags}

We now apply the results above to realize our initial goal of constructing two families of semitoric degenerations of flag varieties. This is done by extending the approach in Subsection~\ref{grassmannians}. 

Fix an integer $n\ge 1$ and a subset $\bd=\{d_0<\dots<d_l\}$ of $[0,n]$ with $d_0=0$ and $d_l=n$. Consider the product \[\bP_\bd=\bP(\wedge^{d_0}\bC^n)\times\dots\times\bP(\wedge^{d_l}\bC^n),\] the first and last factors are zero-dimensional and are included for consistency. The multihomogeneous coordinate ring $S_\bd$ of $\bP_\bd$ is the polynomial ring in variables $X_{i_1,\dots,i_k}$ with $1\le i_1<\dots<i_k\le n$ and $k\in\bd$. The zero set of the Pl\"ucker ideal $\tilde I_\bd\subset S_\bd$ is the partial flag variety $F_\bd$ of increasing chains of subspaces in $\bC^n$ with dimensions $d_0,\dots,d_l$.

Consider the ideal $\tilde I_\bd^{GT}\subset S_\bd$ generated by all binomials of the form 
\begin{equation}
X_{a_1,\dots,a_k}X_{b_1,\dots,b_l}-X_{\min(a_1,b_1),\dots,\min(a_l,b_l)}X_{\max(a_1,b_1),\dots,\max(a_l,b_l),a_{l+1},\dots,a_k}
\end{equation}
where $k\ge l$.
\begin{theorem}[\cite{GL}]\label{glfull}
The ideal $\tilde I_\bd^{GT}$ is an initial ideal of $\tilde I_\bd$.
\end{theorem}
This means that the zero set of $\tilde I_\bd^{GT}$ in $\bP_\bd$ is a flat toric degeneration of $F_\bd$. It is shown in~\cite{KM} that this zero set is, in fact, the toric variety of the corresponding Gelfand--Tsetlin polytope, hence ``GT''.

We also consider the poset $(P_\bd,<)$. Here 
\begin{equation}\label{pbd}
P_\bd=\{p_{d_0+1,d_1},\dots,p_{d_{l-1}+1,d_l}\}\cup\bigcup_{i=1}^{l-1} P_{d_i}
\end{equation}
with the sets $P_{d_i}$ defined as in Subsection~\ref{grassmannians}. Again, we set $p_{i_1,j_1}\le p_{i_2,j_2}$ if and only if $i_1\le i_2$ and $j_1\le j_2$. We use the notation $\tilde p_i=p_{d_{i-1}+1,d_i}$ for the first $l$ elements in right-hand side of~\eqref{pbd}. An example of such a poset can be seen in Example~\ref{gtexample}, in terms of this visualization $\tilde p_i$ is the lowest element lying ``between'' the sets $P_{i-1}$ and $P_i$.

We define a map $\psi_{GT}:S_\bd\to\bC[\cJ(P_\bd,<)]$ by setting 
\begin{equation}\label{psiGT}
\psi_{GT}(X_{a_1,\dots,a_k})=X_{\{p_{i,j}\mid i\le k\text{ and }j\le a_{k+1-i}+i-1\}}.
\end{equation}
In other words, one considers the ideal $L\in\cJ(P_k,<)$ for which $\psi_\cO(X_{a_1,\dots,a_k})=X_L$ (in the notations of Subsection~\ref{grassmannians}) and sets $\psi_{GT}(X_{a_1,\dots,a_k})=X_J$ where $J$ consists of $L$ and all $p_{i,j}\in P_\bd$ with $j\le k$. In particular, if $k=d_j$, then $J$ will contain precisely those $\tilde p_i$ for which $i\le j$.
\begin{example}\label{gtexample}
Let $n=5$ and $\bd=\{0,2,4,5\}$. Below is the Hasse diagram of $(P_\bd,<)$ with the elements $\tilde p_i$ highlighted in red. The order ideal $J$ for which $\psi_{GT}(X_{3,5})=X_J$ consists of the cyan elements together with $p_{1,2}=\tilde p_1$.
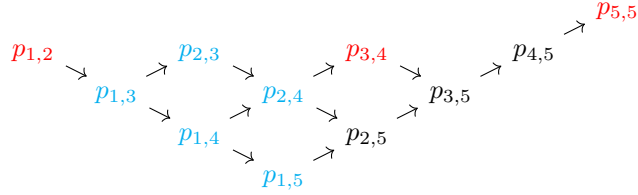
\[
\begin{tikzcd}[row sep=.5mm,column sep=2.5mm]
&& && && &&\color{red} p_{5,5}\\
&{\color{red}p_{1,2}}\arrow[dr]&&\color{cyan}p_{2,3}\arrow[dr] && \color{red} p_{3,4}\arrow[dr]&& p_{4,5}\arrow[ur]\\
&&{\color{cyan}p_{1,3}}\arrow[ur]\arrow[dr]&&\color{cyan}p_{2,4}\arrow[dr]\arrow[ur] && p_{3,5}\arrow[ur]\\
&&&{\color{cyan}p_{1,4}}\arrow[dr]\arrow[ur]&&p_{2,5}\arrow[ur]\\
&&&&\color{cyan}p_{1,5}\arrow[ur]
\end{tikzcd}
\]
\end{example}

\begin{proposition}\label{psiGTisom}
$\psi_{GT}$ is an isomorphism and $\psi_{GT}(\tilde I_\bd^{GT})=I^h_{P_\bd,<}$.
\end{proposition}
\begin{proof}
From the above description of $\psi_{GT}$ and Proposition~\ref{psiOisom} we see that $\psi_{GT}$ establishes a bijection between variables of the form $X_{a_1,\dots,a_{d_j}}$ and $X_J$ for which $X_J$ contains exactly $j$ of the elements $\tilde p_i$. The first claim follows.

Consider $X_{a_1,\dots,a_k},X_{b_1,\dots,b_l}\in S_\bd$ with $k\ge l$ and let $\psi_{GT}(X_{a_1,\dots,a_k})=X_{J_1}$ and $\psi_{GT}(X_{b_1,\dots,b_l})=X_{J_2}$. It is straightforward from the definition that \[\psi_{GT}(X_{\min(a_1,b_1),\dots,\min(a_l,b_l)})=X_{J_1\cap J_2}\] and \[\psi_{GT}(X_{\max(a_1,b_1),\dots,\max(a_l,b_l),a_{l+1},\dots,a_k})=X_{J_1\cup J_2}.\] This proves the second statement.
\end{proof}


Set $P_\bd^*=\{\tilde p_1,\dots,\tilde p_l\}$ and for $i\in[0,l]$ set $K_i=\{\tilde p_1,\dots,\tilde p_i\}$. Consider \[\la=\om_{K_1}+\dots+\om_{K_l}\in\bR^{P_\bd^*}\] (so that $\la_{\tilde p_i}=i$) and the MRPP $\cR_\la(P,<,<_\varnothing)$ for the trivial order $<_\varnothing$. This MRPP is evidently standard and in the notations of Subsection~\ref{MRPPtoric} we have \[D(\la)=\{K_0,\dots,K_l\},\] $\cJ_\la=\cJ(P_\bd,<)$ and $I_\la=I^h_{P_\bd,<}$.

Note that the monomial ideal $\tilde I^m_1=\psi_{GT}^{-1}(I_{P_\bd,<}^m)$ is an initial ideal of $\tilde I^{GT}_\bd$ and, subsequently, of $\tilde I_\bd$. We may now combine Corollary~\ref{standardmain} with Proposition~\ref{psiGTisom} to describe all Gr\"obner degenerations of $F_\bd$ intermediate between the toric degeneration given by $\tilde I^{GT}_\bd$ and the monomial degeneration given by $\tilde I^m_1$. Here we view $\cR_\la(P,<,<_\varnothing)$ as the Gelfand--Tsetlin polytope corresponding to the $\mathfrak{gl}_n$-weight $\om_{d_1}+\dots+\om_{d_{l-1}}$, see Corollary~\ref{gtfflv}.
\begin{theorem}\label{mainGT}
Consider $w\in\overline{C(\tilde I^{GT}_\bd,\tilde I^m_1)}$. There exist orders $<_1,\dots,<_m$ on $P_\bd$ stronger than $<$ such that the zero set of $\initial_w \tilde I^{GT}_\bd$ in $\bP_\bd$ is semitoric with irreducible components isomorphic to the $H_\la(P_\bd,<_i,<_\varnothing)$ and the polytopes $\cR_\la(P_\bd,<_i,<_\varnothing)$ form a polyhedral subdivision of the Gelfand--Tsetlin polytope $\cR_\la(P_\bd,<,<_\varnothing)$.
\end{theorem}


We move on to the second family of degenerations. Let the order $<'$ on $P_\bd$ be such that $p<'q$ if and only if $p<q$ and $p\notin P_\bd^*$. 
We define another map $\psi_{FFLV}$ (Feigin--Fourier--Littelmann--Vinberg, see below) from $S_\bd$ to $\bC[\cJ(P_\bd,<)]$. Consider $J\in\cJ(P_\bd,<)$ and let $J\cap P_\bd^*=K_j$, then $J$ consists of $J\cap P_{d_j}$ (which is an order ideal in $(P_{d_j},<)$) and all $p_{i,j}\in P_\bd$ with $j\le d_j$. We have a unique $X_{\alpha_1,\dots,\alpha_{d_j}}$ for which $\psi_{\cC}(X_{\alpha_1,\dots,\alpha_{d_j}})=X_{J\cap P_{d_j}}$ (see Subsection~\ref{grassmannians}), set $\psi_{FFLV}(X_{\alpha_1,\dots,\alpha_{d_j}})=X_J$. The fact that $\psi_\cC$ is an isomorphism from $S_k$ to $\bC[\cJ(P_k,<)]$ for any $k$ immediately implies that $\psi_{FFLV}$ is itself an isomorphism from $S_\bd$ to $\bC[\cJ(P_\bd,<)]$.
\begin{example}
Again let $n=5$ and $\bd=\{0,2,4,5\}$. The order ideal $J$ for which $\psi_{FFLV}(X_{1,5,3,4})=X_J$ consists of the cyan elements below together with $p_{1,2}=\tilde p_1$ and $p_{3,4}=\tilde p_2$.
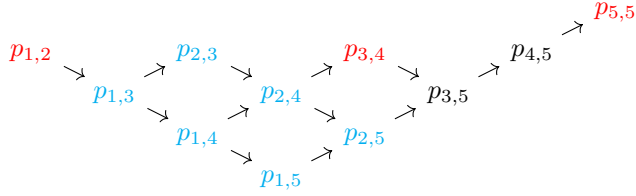
\[
\begin{tikzcd}[row sep=.5mm,column sep=2.5mm]
&& && && &&\color{red} p_{5,5}\\
&{\color{red}p_{1,2}}\arrow[dr]&&\color{cyan}p_{2,3}\arrow[dr] && \color{red} p_{3,4}\arrow[dr]&& p_{4,5}\arrow[ur]\\
&&{\color{cyan}p_{1,3}}\arrow[ur]\arrow[dr]&&\color{cyan}p_{2,4}\arrow[dr]\arrow[ur] && p_{3,5}\arrow[ur]\\
&&&{\color{cyan}p_{1,4}}\arrow[dr]\arrow[ur]&&\color{cyan}p_{2,5}\arrow[ur]\\
&&&&\color{cyan}p_{1,5}\arrow[ur]
\end{tikzcd}
\]
\end{example}

Consider the ideals $I_{P_\bd,<,<'}$ and $\tilde I^{FFLV}_\bd=\psi_{FFLV}^{-1}(I_{P_\bd,<,<'})$. The following fact is essentially due to~\cite{FFL2} and~\cite{fffm}.

\begin{theorem}
$\tilde I^{FFLV}_\bd$ is an initial ideal of $\tilde I_\bd$.
\end{theorem}
\begin{proof}
The definition of $\psi_{FFLV}$ can be reformulated as follows. Consider a subset of $[1,n]$ with $d_j$ elements and let $\alpha_1,\dots,\alpha_{d_j}$ be the ordering of this subset for which $\alpha_i=i$ whenever $\alpha_i\le d_j$ while the remaining elements are ordered decreasingly. Then $\psi_{FFLV}(X_{\alpha_1,\dots,\alpha_{d_j}})=X_J$ where $J$ is the order ideal generated by the $p_{i,\alpha_i}$ for which $\alpha_i>d_j$ and the elements $\tilde p_1,\dots,\tilde p_j$. Moreover, the specified generating set is precisely $\max_{<'}J$. Therefore, in view of~\eqref{relkernel}, $\tilde I^{FFLV}_\bd$ is the kernel of the map from $S_\bd$ to $\bC[P_\bd,t]$ taking such a $X_{\alpha_1,\dots,\alpha_k}$ to \[t\prod_{\alpha_i>d_j} z_{p_{i,\alpha_i}}\prod_{i= 1}^j z_{\tilde p_i}.\] The fact that the kernel of such a map is an initial ideal of the Pl\"ucker ideal is proved in~\cite[Theorem 5.1]{fffm} (modulo a minor change of variables). It is also explained there that the zero set of this kernel in $\bP_\bd$ is the toric variety of the corresponding FFLV polytope, hence the notation.
\end{proof}

Consider the monomial ideal $\tilde I^m_2=\psi_{FFLV}^{-1}(I_{P_\bd,<}^m)$, it is is an initial ideal of $\tilde I^{FFLV}_\bd$ and $\tilde I_\bd$. Similarly to Theorem~\ref{mainGT} we describe all degenerations intermediate between the toric and the monomial one (again, here we view $\cR_\la(P,<,<')$ as the FFLV polytope corresponding to $\om_{d_1}+\dots+\om_{d_{l-1}}$, see Corollary~\ref{gtfflv}).
\begin{theorem}\label{mainFFLV}
Consider $w\in\overline{C(\tilde I^{FFLV}_\bd,\tilde I^m_2)}$. There exist orders $<_1,\dots,<_m$ on $P_\bd$ stronger than $<$ such that the zero set of $\initial_w \tilde I^{FFLV}_\bd$ in $\bP_\bd$ is semitoric with irreducible components isomorphic to the $H_\la(P_\bd,<_i,<')$ and the polytopes $\cR_\la(P_\bd,<_i,<')$ form a polyhedral subdivision of the FFLV polytope $\cR_\la(P_\bd,<,<')$.
\end{theorem}

Note that in accordance with Remark~\ref{chainorder} all polytopes appearing in Theorem~\ref{mainGT} are marked order polytopes but the polytopes in Theorem~\ref{mainFFLV} are more general MRPPs.

\subsection{Comparison with Fang--Fourier--Litza--Pegel polytopes}

In this subsection we compare MRPPs with the polytopes studied in~\cite{FF,FFLP,FFP}. Our goal is to show that our family is more general in an essential way: the earlier construction is insufficient for our description of semitoric degenerations of flag varieties.

Consider a finite poset $(P,<)$ and a subset of marked elements $P^*\subset P$ containing all minimal and maximal elements. We also choose a dominant marking $\la\in\bZ^{P^*}$ and a partition $C\sqcup O=P\backslash P^*$. Again denote $\cJ=\cJ(P,<)$ and $\cJ^*=\cJ(P^*,<)$. In the following definition we dualize the poset, i.e.\ in~\cite{FFLP} the same polytope corresponds to the dual poset $(P,>)$ together with the data $C$, $O$, $\la$.
\begin{definition}[{\cite[Proposition 1.3]{FFLP}}]\label{mcopdef}
The \textit{marked chain-order polytope} (\textit{MCOP}) $\cO_{C,O}(P,<,\la)\subset\bR^P$ consists of points $x$ such that:
\begin{enumerate}
\item $x_p=\la_p$ for all $p\in P^*$,
\item $x_p\ge 0$ for all $p\in C$,
\item for every chain $a<p_1<\dots<p_r<b$ with $a,b\in P^*\sqcup O$ and all $p_i\in C$ one has \[x_{p_1}+\dots+x_{p_r}\le x_a-x_b.\]
\end{enumerate}
\end{definition}

We note that in~\cite{FFLP} the chain in (3) is required to be saturated but one easily checks that this implies the statement for any such chain (see footnote 2 in~\cite{FFP}). Before directly comparing MCOPs with MRPPs we show that the former have some similar properties. 
The following fact similar to Definition~\ref{markeddef} is obtained from~\cite{FFP} by noting that order ideals in $(P,<)$ are precisely order filters in $(P,>)$.

\begin{proposition}[{\cite[Definition 2.4 and Theorem 2.8]{FFP}}]\label{elementarydecomp}
There exists a strictly decreasing chain in $\cJ^*$ \[P^*=K_0\supset\dots\supset K_k\neq\varnothing\] and integers $c_0,\dots,c_k$ such that $\la=\sum_{i=0}^kc_i\om_{K_i}$. These $K_i$ and $c_i$ satisfy \[\cO_{C,O}(P,<,\la)=c_0\cO_{C,O}(P,<,\om_{K_0})+\dots+c_k\cO_{C,O}(P,<,\om_{K_k}).\]
\end{proposition}

Furthermore, the MCOP given by a fundamental marking admits a description similar to Definition~\ref{funddef}.
\begin{proposition}\label{elementary}
For $K\in\cJ^*$ the polytope $\cO_{C,O}(P,<,\om_K)$ is the convex hull of all points of the form $\mathbf 1_{A(J)}$ where $J\in\cJ$ satisfies $J\cap P^*=K$ and \[A(J)=(J\cap (P^*\cup O))\cup\max\nolimits_<J.\]
\end{proposition}
\begin{proof}
To see that $\mathbf 1_{A(J)}\in \cO_{C,O}(P,<,\om_K)$ note that conditions (1) and (2) in Definition~\ref{mcopdef} are obvious for $x=\mathbf 1_{A(J)}$ while condition (3) is verified as follows. If $x_a=x_b=1$, i.e.\ $a,b\in J$, then none of the $p_i$ are maximal in $J$ and both sides of the inequality are 0. If $x_a=x_b=0$, i.e.\ $a,b\notin J$, then no $p_i$ lies in $J$ and again both sides are 0. Otherwise we have $x_a=1$ and $x_b=0$ while the left-hand side cannot be greater than 1 because there is at most one element of $\max_<J$ among the $p_i$.

It is also evident that the $\one_{A(J)}$ are pairwise distinct since $A(J)$ generates the ideal $J$. Furthermore, by~\cite[Proposition 2.4]{FFLP} $\cO_{C,O}(P,<,\om_K)$ is a lattice polytope and it remains to show that it has no integer points other than the $\one_{A(J)}$. However, by~\cite[Corollary 2.5]{FFLP} $\cO_{C,O}(P,<,\om_K)$ has the same number of integer points as $\cO_{\varnothing,P\backslash P^*}(P,<,\om_K)$. The latter polytope consists of $x$ such that $x_p\ge x_q$ whenever $p<q$ and $x_p=(\om_K)_p$ for $p\in P^*$. The integer points in such a polytope are seen to be the $\one_J$ with $J\cap P^*=K$.
\end{proof}




Consider an order $<'$ on $P$ such that $p<'q$ if and only if $p<q$ and $p\notin P^*\cup O$. The following fact lets us consider the MRPP $\cR_\la(P,<,<')$.
\begin{proposition}
The set $\cJ$ is closed under $*_{P,<'}$.
\end{proposition}
\begin{proof}
For $J_1,J_2\in\cJ$ suppose that \[J=J_1*_{P,<'}J_2\notin\cJ.\] This means that we have $p\notin J$ and $q\in J$ for some $p<q$. Since $J\in\cJ(P,<')$, this means that $p\in P^*\cup O$. By~\eqref{star} $J$ is contained in $J_1\cap J_2$, hence $q,p\in J_1\cap J_2$. Note that $P^*\cup O\subset\max_{<'} P$ and, therefore, $p$ is $<'$-maximal in every containing subset. Thus formula~\eqref{star} provides $p\in J$ contradicting our assumption.
\end{proof}

We can now easily deduce that MRPPs are at least as general as MCOPs.
\begin{theorem}\label{mcopasmrpp}
$\cO_{C,O}(P,<,\la)=\cR_\la(P,<,<')$.
\end{theorem}
\begin{proof}
First suppose that $\la=\om_F$ is fundamental. Then the claim is immediate from comparing Definition~\ref{funddef} with Proposition~\ref{elementary}.

Now consider a general $\la$ together with the $K_i$ and $c_i$ given by Proposition~\ref{elementarydecomp}. One easily checks that for a nonnegative $\la$ all $c_i$ satisfy $c_i=\alpha_{K_i}(\la)$. For other $\la$ the only negative $c_i$ is $c_0$, i.e.\ $\la+c_0\om_{P^*}$ is nonnegative and $c_i=\alpha_{K_i}(\la+c_0\om_{P^*})$ for $c_i>0$. In both cases the claim follows by comparing Definition~\ref{markeddef} with Proposition~\ref{elementarydecomp}.
\end{proof}

\begin{remark}
It should be mentioned that in~\cite{FFLP,FFP} MCOPs are studied within a larger continuous family termed \textit{marked poset polytopes}. However, the key combinatorial properties such as integrality, normality, pairwise Ehrhart-equivalence, etc.\ are proved in the generality of Definition~\ref{mcopdef} and do not hold for general marked poset polytopes.
\end{remark}

\begin{remark}
Some of the above results can be applied to deduce curious new properties of MCOPs. In particular, using Theorem~\ref{mrppmultiproj} one obtains a multiprojective realization of an MCOP's toric variety and using Theorem~\ref{standard} one deduces that every MCOP is unimodularly equivalent to another MCOP $\cO_{C,O}(Q,\ll,\mu)$ with $(Q^*,\ll)$ linearly ordered. Furthermore, Theorem~\ref{mrppmain} can be applied to the MRPP $\cR_\la(P,<,<')$ providing a family of semitoric degenerations of the toric variety of the coinciding MCOP $\cO_{C,O}(P,<,\la)$ with components given by MRPPs.
\end{remark}

Note that when $C=\varnothing$ the polytope $\cO_{C,O}(P,<,\la)$ is given by the inequalities $x_p\ge x_q$ for $p<q$. When $O=\varnothing$ the polytope is given by $x_p\ge 0$ for $p\in C$ and \[x_{p_1}+\dots+x_{p_r}\le \la_b-\la_a\] for any chain $a<p_1<\dots<p_r<b$ with $a,b\in P^*$ and $p_i\in C$. In particular, for $(P_\bd,<)$, $P_\bd^*$ and $\la$ as in the previous subsection the polytopes $\cO_{\varnothing,O}(P_\bd,<,\la)$ and $\cO_{C,\varnothing}(P_\bd,<,\la)$ given by these inequalities are known, respectively, as the Gelfand--Tsetlin polytope~(\cite{GT}) and the FFLV polytope~(\cite{FFL1}) of the $\mathfrak{gl}_n$-weight $\om_{d_1}+\dots+\om_{d_{l-1}}$ where $\om_i$ denotes the $i$th fundamental weight (see also~\cite{ABS}). We point out that it is more standard to consider Gelfand--Tsetlin and FFLV polytopes inside a larger space with coordinates labeled by \textit{all} pairs $1\le i\le j\le n$ but in both cases coordinates corresponding to $i,j$ with $p_{i,j}\notin P_\bd$ are constant throughout the polytope. Hence, Theorem~\ref{mcopasmrpp} provides
\begin{cor}\label{gtfflv}
In the notations of Subsection~\ref{flags}, the polytope $\cR_\la(P_\bd,<,<_\varnothing)$ is the Gelfand--Tsetlin polytope of the $\mathfrak{gl}_n$-weight $\om_{d_1}+\dots+\om_{d_{l-1}}$ and $\cR_\la(P_\bd,<,<')$ is the FFLV polytope of this weight.
\end{cor}

For general $(P,<)$ and $\la$ the set of MRPPs of the form $\cR_\la(P,<,<')$ is strictly larger than the set of MCOPs of the form $\cO_{C,O}(P,<,\la)$. This can already be seen for posets with five elements. More importantly for us, MRPPs that are not MCOPs arise in the descriptions of semitoric degenerations of flag varieties given by Theorem~\ref{mainFFLV}. This is best seen in the Grassmannian case corresponding to $\bd=\{0,k,n\}$. In this case $\la=\om_{K_k}$ is fundamental and $\cJ(P_\bd,<)$ is in bijection with the vertex set of the FFLV polytope \[\Pi_{k,n}=\cR_\la(P_\bd,<,<')=\cO_{P_\bd\backslash P^*,\varnothing}(P_\bd,<,\la).\] The subdivision considered in the theorem is found directly from Zhu's theorem as $\Theta_{\Pi_{k,n}}(w)$. Since a part in this subdivision has the form $\cR_\la(P_\bd,<'',<')$, its vertices are parametrized by $\cJ(P_\bd,<'')$ and, in view of Proposition~\ref{elementary}, it is natural to ask whether this part actually has the form $\cO_{C,O}(P_\bd,<'',\la)$ for some $C$ and $O$. This turns out to be false in general.

\begin{example}\label{notmcop}
Let $n=5$ and $\bd=\{0,2,5\}$ corresponding to the Grassmannian $\Gr(2,5)$. Choose $w\in\bR^{\cJ(P_\bd,<)}$ with $w_{J_1}=1$ for $J_1=\{\tilde p_1,p_{1,3},p_{1,4},p_{1,5}\}$ and the other 9 coordinates $w_J=0$. Then $\Theta_{\Pi_{2,5}}(w)$ consists of two parts: a polytope $Q$ with 9 vertices (those for which $w_J=0$) and a simplex. The vertices of $Q$ correspond to $J\in \cJ(P_5,<'')$ where $<''$ is obtained from $<$ by adding the relation $p_{2,3}<''p_{1,5}$. One may check that $Q$ does not have the form $\cO_{C,O}(P_\bd,<'',\la)$ for any $C$ and $O$ which motivates us to search for a larger family of poset polytopes associated with $(P_\bd,<'')$.
\end{example}

\end{document}